\documentclass[10pt,draftcls,onecolumn]{IEEEtran}
\usepackage{cite}
\usepackage{amsmath,amssymb,amsfonts}
\usepackage{multirow}
\usepackage{amsthm}
\usepackage{algorithmic}
\usepackage{graphicx}
\usepackage{textcomp}
\usepackage{epsfig}
\usepackage{algorithmic}
\usepackage{epstopdf}
\usepackage[ruled]{algorithm}
\usepackage{enumerate}
\usepackage{float}
\usepackage{graphics,graphicx}
\usepackage{subfigure}
\usepackage{array}

\usepackage[justification=centering]{caption}

\DeclareGraphicsExtensions{.png}
\graphicspath{{./Pics/}}
\usepackage{color}
\usepackage{xspace}

\def\BibTeX{{\rm B\kern-.05em{\sc i\kern-.025em b}\kern-.08em
    T\kern-.1667em\lower.7ex\hbox{E}\kern-.125emX}}

\newtheorem{lemma}{Lemma}
\newtheorem{theorem}{Theorem}

\newtheorem{assumption}{Assumption}

\newtheorem{remark}{Remark}

\title{\Large \bf Decentralized Approximate Newton Methods for Convex Optimization on Networked Systems}

\author{Hejie Wei, Zhihai Qu, Xuyang Wu, Hao Wang, and Jie Lu
\thanks{This work has been supported by the National Natural Science Foundation of China under grant 61603254 and the Natural Science Foundation of Shanghai under grant 16ZR1422500.}
\thanks{Hejie Wei is with the School of Statistics and Mathematics, Shanghai Lixin University of Accounting and Finance, 201209 Shanghai, China (e-mail: 20190057@lixin.edu.cn). Zhihai Qu, Xuyang Wu, Hao Wang, and Jie Lu are with the School of Information Science and Technology, ShanghaiTech University, 201210 Shanghai, China (e-mail: \{quzhh1, wuxy, wanghao1, lujie\}@shanghaitech.edu.cn).}}

\begin{document}

\maketitle
\thispagestyle{empty}
\pagestyle{empty}

\begin{abstract}
In this paper, a class of Decentralized Approximate Newton (DEAN) methods for addressing convex optimization on a networked system are developed, where nodes in the networked system seek for a consensus that minimizes the sum of their individual objective functions through local interactions only. The proposed DEAN algorithms allow each node to repeatedly perform a local approximate Newton update, which leverages tracking the global Newton direction and dissipating the discrepancies among the nodes. Under less restrictive problem assumptions in comparison with most existing second-order methods, the DEAN algorithms enable the nodes to reach a consensus that can be arbitrarily close to the optimum. Moreover, for a particular DEAN algorithm, the nodes linearly converge to a common suboptimal solution with an explicit error bound. Finally, simulations demonstrate the competitive performance of DEAN in convergence speed, accuracy, and efficiency.
\end{abstract}

\begin{IEEEkeywords}
Distributed optimization, network optimization, second-order methods
\end{IEEEkeywords}
\section{INTRODUCTION}

\label{sec:introduction}

In many engineering applications such as network state estimation \cite{Rabbat04}, network resource allocation \cite{Beck14}, and distributed learning \cite{Bach14}, nodes in a networked system often need to cooperate with each other in order to minimize the sum of their individual objective functions.

There have been a large number of discrete-time decentralized/distributed algorithms for such in-network optimization problems, which allow the nodes to address the problem by means of interacting with their neighbors only. Most of these algorithms are first-order methods, where the nodes utilize subgradients/gradients of their local objectives to update (e.g., \cite{Rabbat04, Nedic10, ZhuM12, Beck14, ShiW15, LinP16, YuanK16, QuG17, Nedic17, XiC18, ShiC18}). However, the first-order algorithms may suffer from slow convergence, especially when the problem is ill-conditioned. This motivates the development of decentralized second-order methods, where the Hessian matrices of certain objectives are involved in computing the iterates so that the convergence may be accelerated. 

The existing decentralized second-order methods can be roughly classified into two categories. The first category rely on second-order approximations of certain dual-related objectives. For instance, the decentralized Exact Second-Order Method (ESOM) \cite{Mokhtari16A} considers a second-order approximation of an augmented Lagrangian function, and the Decentralized Quadratically Approximated ADMM (DQM) \cite{Mokhtari16B} introduces a quadratic approximation to a decentralized version of the Alternating Direction Method of Multipliers (ADMM). The Primal-Dual Quasi-Newton method (PD-QN) \cite{Eisen19} performs quasi-Newton updates on both primal and dual variables to optimize an augmented Lagrangian function. The SDD-solver-based Newton algorithm proposed by \cite{Tutunov19} takes advantage of the sparsity of the dual Hessian and approximates the Newton directions up to arbitrary accuracy through integrating a solver for symmetric diagonally dominant (SDD) linear equations.

The second category is the Newton-type methods, such as the Decentralized Broyden-Fletcher-Goldfarb-Shanno (D-BFGS) method \cite{Eisen17}, the Network Newton (NN) method \cite{Mokhtari17}, the Asynchronous Network Newton method (ANN) \cite{Mansoori19}, the Distributed Quasi-Newton (DQN) method \cite{Bajovic15}, and the Newton-Raphson Consensus (NRC) method \cite{Varagnolo16}. Among these methods, D-BFGS, NN, ANN, and DQN employ a penalized objective function in order to relax a consensus constraint and approximate the Newton direction of the penalized objective in a decentralized manner. As a result, these methods converge to a suboptimal solution. NRC utilizes an average consensus scheme to approximate the Newton-Raphson direction in a distributed fashion. Although NRC may converge to the exact optimal solution, no explicit parameter condition for convergence is provided.



In this paper, we propose a family of discrete-time \underline{De}centralized \underline{A}pproximate \underline{N}ewton methods, referred to as DEAN, for solving convex optimization on networked systems. In the DEAN algorithms, each node evolves along a local approximate Newton direction, which is jointly determined by the Hessian inverse of the node's local objective function and a possibly nonlinear consensus operation. This combines the approximation of the global Newton direction with the dissipation of the nodes' disagreements. Moreover, a special case of DEAN can be viewed as a finite-difference discretization of the continuous-time ZGS algorithms in \cite{LuJ12}, yet DEAN can tackle more general problems and the convergence analysis of DEAN is significantly different from that in \cite{LuJ12} (cf. Remark~\ref{rem:ZGS}). 

We show that the DEAN algorithms asymptotically drive all the nodes to a consensus that can lie in an \emph{arbitrarily small} neighborhood of the optimal solution, when solving a class of \emph{locally} strongly convex problems with \emph{locally} Lipschitz objective gradients. We also show that a particular DEAN algorithm achieves a \emph{linear rate} of convergence to a suboptimal solution with an explicit error bound given. Note that the linear convergence rate of DEAN is derived under less restrictive or different problem conditions compared to most prior first-order methods \cite{ShiW15,YuanK16,Nedic17} and second-order methods \cite{Mokhtari16A,Mokhtari16B,Mokhtari17,Bajovic15,Varagnolo16,LuJ12,Eisen17,Eisen19,Tutunov19,Mansoori19}. In particular, these existing works assume global strong convexity of the objective functions and global Lipschitz continuity of the objective gradients to establish linear convergence, which are unnecessary for DEAN. Furthermore, simulation results illustrate the superior performance of DEAN in convergence speed, accuracy, and computational efficiency for addressing logistic regression in comparison with a number of start-of-the-art first-order and second-order methods.


The outline of this paper is as follows: Section~\ref{sec:probform} formulates the problem. Section~\ref{sec:AZGS} describes the proposed DEAN algorithms and Section~\ref{sec:convanal} is dedicated to the convergence analysis. Section~\ref{sec:numeexam} presents the simulation results. Concluding remarks are provided in Section~\ref{sec:conclusion}. All the proofs are in the appendix.

A preliminary, $5$-page conference version of this paper can be found in \cite{WeiH18}, which contains no proof. This paper significantly expands \cite{WeiH18} by including new convergence rate analysis, new numerical study, more detailed discussions on algorithm design, theoretical results, and comparisons with existing works, as well as all the proofs. 

\subsection{Notation} 

We use $\|\cdot\|$ to denote the Euclidean norm, $\{\cdot,\cdot\}$ the unordered pair, $|\cdot|$ the absolute value of a real number or the cardinality of a set, and $\text{range}(\cdot)$ the range of a matrix. In addition, we use $I_n$ to represent the $n\times n$ identity matrix and $0$ to represent a zero matrix of proper size. For any $z_1,\ldots,z_N\in \mathbb{R}^n$, $\mathbf{z}=[z_1;\dots;z_N]\in\mathbb{R}^{nN}$ is the vector obtained by stacking $z_1,\ldots,z_N$. For any $c_1,\ldots,c_n\in\mathbb{R}$, $\text{diag}(c_1,\ldots,c_n)$ is the diagonal matrix whose diagonal entries are $c_1,\ldots,c_n$. For any $A, B\in \mathbb{R}^{n\times n}$, $A\succeq B$ means $A-B$ is positive semidefinite and $A\succ B$ means $A-B$ is positive definite. For any symmetric positive semidefinite matrix $H\in\mathbb{R}^{n\times n}$, we use $\lambda_i(H)$ to denote the $i$th smallest eigenvalue of $H$, $\lambda_{\max}(H)$ the largest eigenvalue of $H$, and $H^{\dagger}$ the pseudoinverse of $H$. For any set $\mathcal{C}\subseteq\mathbb{R}^n$, $\operatorname{conv}\mathcal{C}$ is the convex hull of $\mathcal{C}$ and $P_{\mathcal{C}}(x)$ is the projection of $x\in\mathbb{R}^n$ onto $\mathcal{C}$. For any differentiable function $f:\mathbb{R}^n\rightarrow\mathbb{R}$, $\nabla f(x)$ denotes the gradient of $f$ at $x\in\mathbb{R}^n$ and, if $f$ is twice differentiable, $\nabla^2 f(x)$ represents the Hessian matrix of $f$ at $x$. 

\section{Problem Formulation}\label{sec:probform}

\subsection{Preliminaries}

Consider a twice continuously differentiable function $f:\mathbb{R}^n\rightarrow\mathbb{R}$. It is said to be \emph{locally strongly convex} if for any convex and compact set $\mathcal{C}$, there exists $\theta_{\mathcal{C}}>0$ such that 
\begin{align*}
\nabla^2 f_i(x)&\succeq\theta_{\mathcal{C}}I_n,\;\forall x\in \mathcal{C}.
\end{align*}
This also indicates that
\begin{align*}
f(y)-f(x)-\nabla f(x)^T(y-x)&\ge\theta_{\mathcal{C}}\|y-x\|^2/2,\;\forall x,y\in \mathcal{C},\displaybreak[0]\\
(\nabla f(y)-\nabla f(x))^T(y-x)&\ge\theta_{\mathcal{C}}\|y-x\|^2,\;\forall x,y\in \mathcal{C}.
\end{align*}
Here, $\theta_{\mathcal{C}}>0$ is called the \emph{convexity parameter} of $f$ on $\mathcal{C}$. We say $f$ is \emph{globally strongly convex} if the above conditions hold for all points in $\mathbb{R}^n$ with a uniform convexity parameter that is positive. 

A vector-valued or matrix-valued function $h$ is said to be \emph{locally Lipschitz continuous} if for any compact set $\mathcal{C}$ contained in the domain of $h$, there exists $L_{\mathcal{C}}\ge0$ such that the Lipschitz condition $\|h(x)-h(y)\|\le L_{\mathcal{C}}\|x-y\|$ holds for all $x,y\in \mathcal{C}$. Also, $L_{\mathcal{C}}$ is said to be the \emph{Lipschitz constant} of $h$ on $\mathcal{C}$. Likewise, if the Lipschitz condition holds for all $x,y$ in the domain with a uniform Lipschitz constant that is nonnegative and finite, then $h$ is said to be \emph{globally Lipschitz continuous}.

\subsection{Optimization on Networked Systems}

We model the networked system as an undirected, connected graph $\mathcal{G}=(\mathcal{V},\mathcal{E})$, where $\mathcal{V}=\{1,2,\dots,N\}$ is the set of $N\ge2$ nodes and $\mathcal{E}\subseteq\{\{i,j\}:i,j\in \mathcal{V},\;i\neq j\}$ is the set of links. For each node $i\in\mathcal{V}$, the set of its neighbors is denoted by $\mathcal{N}_i=\{j\in \mathcal{V}:\{i,j\}\in \mathcal{E}\}$. The nodes need to solve
\begin{align}\label{pro}
\min_{x\in\mathbb{R}^n} \sum_{i\in \mathcal{V}} f_i(x),
\end{align}
where each $f_i:\mathbb{R}^n\rightarrow\mathbb{R}$ is the local objective function of node $i\in\mathcal{V}$ and satisfies the following assumption.

\begin{assumption}\label{f}
For each $i\in \mathcal{V}$, $f_i$ is twice continuously differentiable and locally strongly convex. In addition, $f_i$ has a minimizer and $\nabla^2 f_i$ is locally Lipschitz continuous.
\end{assumption}

Assumption~\ref{f} suggests that each $f_i$ has a unique minimizer $x_i^*\in\mathbb{R}^n$, and there is a unique optimal solution $x^*\in\mathbb{R}^n$ to \eqref{pro}. The twice continuous differentiability of $f_i$ is a standard assumption for second-order methods, which implies $\nabla f_i$ is locally Lipschitz continuous (but not necessarily globally Lipschitz). It is also assumed in the existing second-order methods \cite{Mokhtari16A,Mokhtari16B,Mokhtari17,Bajovic15,Varagnolo16,LuJ12,Tutunov19,Mansoori19}, while the quasi-Newton methods \cite{Eisen17,Eisen19} only require $f_i$ to be twice differentiable.

Among the existing distributed second-order methods, most of them \cite{Mokhtari16A,Mokhtari16B,Mokhtari17,Bajovic15,Varagnolo16,LuJ12,Eisen19,Tutunov19,Mansoori19} assume the $f_i$'s to be globally strongly convex, which is more restricted than the \emph{local} strong convexity in Assumption~\ref{f}. One example of local but not global strong convexity is the objective function of logistic regression \cite{Bach14}, i.e., $\sum_{i=1}^N\sum_{j=1}^{m_i}\ln(1+\operatorname{exp}(-v_{ij} u_{ij}^Tx))$, where each $u_{ij}\in\mathbb{R}^n$ is a random feature vector with given label $v_{ij}\in\{-1,1\}$ and the sample number $m_i$ is a sufficiently large integer. We will discuss the problem assumptions of this paper and the existing works more detailedly in Section~\ref{ssec:convcomp}.

\section{Decentralized Approximate Newton Methods}\label{sec:AZGS}

In this section, we develop a class of decentralized Newton-like algorithms to address problem~\eqref{pro}. 

To solve problem~\eqref{pro} in a decentralized way, we reformulate problem~\eqref{pro} by separating the global objective function and adding a consensus constraint as follows:
\begin{align}
\min_{\mathbf{x}\in\mathbb{R}^{nN}}\;&\;F(\mathbf{x})=\sum_{i\in \mathcal{V}} f_i(x_i)\nonumber
\\
\text{s.t.}\;\;&\;x_i=x_j,\;\forall i,j\in\mathcal{V},\label{eq:equivprob}
\end{align}
where $x_i\in\mathbb{R}^n$ $\forall i\in\mathcal{V}$ and $\mathbf{x}=[x_1;\dots;x_N]\in\mathbb{R}^{nN}$.

We first consider applying the classic Newton method
\begin{equation}\label{alg:newton}
\mathbf{x}^{k+1}=\mathbf{x}^{k}-\alpha(\nabla^2 F(\mathbf{x}^{k}))^{-1}\nabla F(\mathbf{x}^{k}),
\end{equation}
where $\alpha>0$ is the Newton step-size. Apparently, the global Newton direction of $F$ in \eqref{alg:newton} can be computed in parallel by the nodes. However, \eqref{alg:newton} does not take into consideration the consensus constraint $x_i=x_j$ $\forall i,j\in\mathcal{V}$ in \eqref{eq:equivprob}, so that it in general cannot converge to a feasible solution. To overcome this issue, below we design a Newton-like direction which not only approximates the global Newton direction of $F$ but also attempts to follow the consensus constraint. 

To this end, we start with a special case of problem~\eqref{pro} with $f_i(x)=\|x-b_i\|^2/2$, $b_i\in\mathbb{R}^n$, so that the problem reduces to an average consensus problem for seeking $\sum_{i\in\mathcal{V}}b_i/N$. This can be solved by a distributed linear consensus algorithm \cite{Olfati-Saber04}:
\begin{align}
\mathbf{x}^{k+1}=\mathbf{x}^{k}-\alpha(\hat{H}_\mathcal{G}\otimes I_n) \mathbf{x}^{k}.\label{eq:consensus}
\end{align}
Here, $\hat{H}_\mathcal{G}=\hat{H}_\mathcal{G}^T\in\mathbb{R}^{N\times N}$ is a weight matrix given by
\begin{eqnarray}\label{matrix:H_G} 
[{\hat{H}_\mathcal{G}}]_{ij}=\left\{
\begin{array}{ll}
\sum\limits_{s\in\mathcal{N}_i}h_{\{i,s\}}, & \text{if}\; i=j,\\
-h_{\{i,j\}},& \text{if}\; \{i,j\}\in\mathcal{E},\\
0,&\text{otherwise},
\end{array}
\right.
\end{eqnarray} 
with $h_{\{i,j\}}>0$ $\forall\{i,j\}\in\mathcal{E}$. It is shown in \cite[Theorem 4]{Olfati-Saber04} that by letting $\alpha\in(0,1/\max_{i\in\mathcal{V}}[{\hat{H}_\mathcal{G}}]_{ii})$ and $\mathbf{x}^0=[b_1;\ldots;b_N]$, $\lim_{k\rightarrow\infty}\mathbf{x}^k=[\frac{1}{N}\sum_{i\in\mathcal{V}}b_i;\ldots;\frac{1}{N}\sum_{i\in\mathcal{V}}b_i]$.

In the above special case, because $\nabla^2 F(\mathbf{x})\equiv I_{nN}$, the linear consensus algorithm \eqref{eq:consensus} is identical to
\begin{align*}
\mathbf{x}^{k+1}=\mathbf{x}^{k}-\alpha(\nabla^2 F(\mathbf{x}^{k}))^{-1}(\hat{H}_\mathcal{G}\otimes I_n) \mathbf{x}^{k}.
\end{align*}
This can also be viewed as being obtained by replacing the gradient term $\nabla F(\mathbf{x}^{k})$ in the Newton method \eqref{alg:newton} with the linear consensus term $(\hat{H}_\mathcal{G}\otimes I_n) \mathbf{x}^{k}$. By assigning the $i$th $n$-dimensional block $x_i^k$ of $\mathbf{x}^k$ to node $i$, the above equation can be written as
\begin{equation}\label{iter:special}
x_i^{k+1}\!=x_i^{k}\!+\alpha(\nabla^2 f_i({x}_i^{k}))^{-1}\!\sum_{j\in\mathcal{N}_i}\!h_{\{i,j\}}(x_j^{k}-x_i^{k}),\quad\forall i\in\mathcal{V}.
\end{equation}

To shed some light on the possibility of solving more general convex problems via \eqref{iter:special}, note that by rearranging \eqref{iter:special} and utilizing the structure of $\hat{H}_\mathcal{G}$ in \eqref{matrix:H_G},
\begin{align*}
&\sum_{i\in \mathcal{V}}\nabla^2f_i(x_i^k)(x_i^{k+1}-x_i^k)\displaybreak[0]\\
=&\sum_{i\in\mathcal{V}}\nabla^2f_i(x_i^k)\alpha(\nabla^2 f_i({x}_i^{k}))^{-1}\sum_{j\in\mathcal{N}_i}h_{\{i,j\}}(x_j^{k}-x_i^{k})\displaybreak[0]\\
=&\alpha\sum_{\{i,j\}\in\mathcal{E}}\left(h_{\{i,j\}}(x_j^k-x_i^k)+h_{\{i,j\}}(x_i^k-x_j^k)\right)=0,
\end{align*}
so that
\begin{align}
&\|\sum_{i\in\mathcal{V}}\nabla f_i(x_i^{k+1})-\sum_{i\in\mathcal{V}}\nabla f_i(x_i^{k})\|\nonumber\displaybreak[0]\\
=&\|\sum_{i\in\mathcal{V}}\int_0^1\nabla^2f_i(x_i^{k}+s(x_i^{k+1}-x_i^{k}))(x_i^{k+1}-x_i^{k})ds\|\nonumber\displaybreak[0]\\
=&\|\!\sum_{i\in\mathcal{V}}\!\int_0^1\!\!\!\big[\nabla^2\!f_i(x_i^{k}\!+\!s(x_i^{k+1}\!-\!x_i^{k}))\!-\!\!\nabla^2\!f_i(x_i^k)\big]\!(x_i^{k+1}\!-\!x_i^{k})ds\|\nonumber\displaybreak[0]\\
\le&\sum_{i\in\mathcal{V}}\int_0^1\|\nabla^2f_i(x_i^{k}+s(x_i^{k+1}-x_i^{k}))-\nabla^2f_i(x_i^{k})\|ds\cdot\|x_i^{k+1}-x_i^{k}\|.
\label{eq:gradientsumchange}
\end{align}
It can be seen from \eqref{eq:gradientsumchange} that if each $f_i$ is a positive definite quadratic function, i.e., 
\begin{align}
\!\!\!f_i(x)=\frac{(x-b_i)^TB_i(x-b_i)}{2},\;B_i=B_i^T\succ 0,\;b_i\in\mathbb{R}^n,\label{eq:quadratic}
\end{align}
then $\sum_{i\in\mathcal{V}}\nabla f_i(x_i^{k})$ remains constant. Under \eqref{eq:quadratic}, if we set 
\begin{align}
x_i^0=x_i^*:=\operatorname{arg\;min}_{x\in\mathbb{R}^n}f_i(x),\label{eq:xi0}
\end{align}
which guarantees $\sum_{i\in\mathcal{V}}\nabla f_i(x_i^{0})=0$, then $\sum_{i\in\mathcal{V}}\nabla f_i(x_i^{k})\equiv 0$. If, in addition, the consensus operation in \eqref{iter:special} makes the $x_i^k$'s asymptotically reach a consensus, then the consensus is exactly the optimum $x^*$.

When $f_i$ is generalized to a non-quadratic convex function, $\sum_{i\in\mathcal{V}}\nabla f_i(x_i^{k})\equiv 0$ may not be guaranteed. However, due to \eqref{iter:special} and \eqref{eq:gradientsumchange}, if the $x_i^k$'s remain bounded, then
\begin{align*}
&\|\sum_{i\in\mathcal{V}}\nabla f_i(x_i^{k+1})-\sum_{i\in\mathcal{V}}\nabla f_i(x_i^{k})\|\le C\sum_{\{i,j\}\in\mathcal{E}}\|x_i^{k}-x_j^{k}\|
\end{align*}
for some constant $C>0$. This indicates that once all the $x_i^k$'s become identical, the gradient sum $\sum_{i\in\mathcal{V}}\nabla f_i(x_i^{k})$ would no longer change. Therefore, if \eqref{eq:xi0} holds and the nodes reach a consensus quickly enough, they could eventually agree on a suboptimal solution that is sufficiently close to the optimum $x^*$. We will validate such intuitive conviction through convergence analysis in Section~\ref{sec:convanal}.

Below, we extend \eqref{iter:special} to a more general form in order to enhance its flexibility and applicability. First, as is suggested in \cite{LuJ12}, to better capture the nonlinearity of the objective gradient in the Newton method, we may extend the linear consensus term $h_{\{i,j\}}(x_j^{k}-x_i^{k})$ in \eqref{iter:special} to a nonlinear one $\nabla g_{\{i,j\}}(x_j^{k})-\nabla g_{\{i,j\}}(x_i^{k})$, where each $g_{\{i,j\}}$, $\{i,j\}\in\mathcal{E}$ is a surrogate function associated with link $\{i,j\}\in\mathcal{E}$ and satisfies Assumption~\ref{gijl} below. 

\begin{assumption}\label{gijl}
For each $\{i,j\}\in \mathcal{E}$, $g_{\{i,j\}}$ is twice continuously differentiable and locally strongly convex. 
\end{assumption}

There are numerous choices of $g_{\{i,j\}}$ satisfying Assumption~\ref{gijl}, which can be either dependent or independent of the objective functions $f_i$'s. For example, for each $\{i,j\}\in\mathcal{E}$, we may let $g_{\{i,j\}}(x)=f_i(x)+f_j(x)$ or $g_{{\{i,j\}}}(x)=\frac{1}{2}x^TA_{\{i,j\}}x$, where $A_{\{i,j\}}\in \mathbb{R}^{n\times n}$ can be any symmetric positive definite matrix known to both node $i$ and node $j$. In practice, we may empirically choose the $g_{\{i,j\}}$'s under Assumption~\ref{gijl} which yield satisfactory convergence performance.

Second, as the Newton step-size $\alpha$ is a global parameter that all the nodes need to agree on, we eliminate such centralized coordination by introducing a local step-size $\alpha_{\{i,j\}}>0$ to each link $\{i,j\}\in\mathcal{E}$, which yields
\begin{align}\label{scheme}
x_i^{k+1}=&x_i^{k}+(\nabla^2 f_i(x_i^{k}))^{-1}\sum_{j\in\mathcal{N}_i}\alpha_{\{i,j\}}(\nabla g_{\{i,j\}}(x_j^{k})-\nabla g_{\{i,j\}}(x_i^{k})),\quad\forall i\in\mathcal{V}.
\end{align}
Equation \eqref{scheme} defines a networked dynamical system, which intends to simultaneously track the global Newton direction of $F$ and facilitate the feasibility of problem~\eqref{eq:equivprob}. This, along with the initialization \eqref{eq:xi0}, yields a class of \emph{Decentralized Approximate Newton methods}, referred to as DEAN.

\begin{remark}\label{rem:ZGS}
In the special case where all the local step-sizes $\alpha_{\{i,j\}}$ $\forall\{i,j\}\in\mathcal{E}$ are identical, DEAN can be viewed as a finite-difference discretization of the continuous-time ZGS algorithms \cite{LuJ12}. Nevertheless, a centralized step-size requires additional coordination throughout the network. Moreover, the convergence analysis of DEAN (cf. Section~\ref{sec:convanal}) is more challenging than and fundamentally different from that of ZGS, because DEAN relaxes the global strong convexity condition required by ZGS to \emph{local} strong convexity and generally does not enjoy the favorable property of ZGS that $\{[y_1;\ldots;y_N]\in\mathbb{R}^{nN}\!\!:\!\sum_{i\in\mathcal{V}}\nabla f_i(y_i)=0\}$ is a positive invariant manifold.
\end{remark}

As is shown in Algorithm~\ref{alg:DEAN}, the implementation of the DEAN algorithms is fully decentralized. The initialization \eqref{eq:xi0} can be completed by the nodes on their own and the update \eqref{scheme} requires each node $i\in\mathcal{V}$ to repeatedly exchange its $x_i^{k}$ with its neighbors. In fact, we can reduce the initialization cost by allowing the existence of small errors in calculating the $x_i^*$'s, so that $\sum_{i\in\mathcal{V}}\nabla f_i(x_i^{0})$ is close to zero rather than exactly zero. This would cause straightforward (but lengthy and tedious) modifications to the analysis in Section~\ref{sec:convanal} and would not fundamentally affect the convergence results.


{
	\renewcommand{\baselinestretch}{1.05}
	\begin{algorithm} [ht]
		\caption{\small Decentralized Approximate Newton (DEAN) Method}
		\label{alg:DEAN}
		\begin{algorithmic}[1]
			\small
			\STATE \textbf{Initialization:}
			\STATE Each pair of neighboring nodes $i$ and $j$ agree on $\alpha_{\{i,j\}}>0$ and $g_{\{i,j\}}$ satisfying Assumption~\ref{gijl}.
			\STATE Each node $i\in\mathcal{V}$ sets $x_i^0=x_i^*$.
			\FOR{ $k=0,1,\ldots$}
			\STATE Each node $i\in \mathcal{V}$ sends $x_i^k$ to every neighbor $j\in\mathcal{N}_i$.
			\STATE Upon receiving $x_j^{k}$ $\forall j\in\mathcal{N}_i$, each node $i\in\mathcal{V}$ updates $x_i^{k+1}=x_i^{k}+(\nabla^2 f_i(x_i^{k}))^{-1}\sum_{j\in\mathcal{N}_i}\alpha_{\{i,j\}}(\nabla g_{\{i,j\}}(x_j^{k})-\nabla g_{\{i,j\}}(x_i^{k}))$.
			\ENDFOR
		\end{algorithmic}
	\end{algorithm}
}

\section{Convergence Analysis}\label{sec:convanal}

In this section, we analyze the convergence performance of the DEAN algorithms. 

Our analysis is based on the Lyapunov function candidate $V: \mathbb{R}^{nN}\to \mathbb{R}$ given by
\begin{align*}
V(\mathbf{x})=\sum_{i\in \mathcal{V}}f_i(x^*)-f_i(x_i)-\nabla f_i(x_i)^T(x^*-x_i).
\end{align*}
Due to Assumption~\ref{f}, $V(\mathbf{x})\ge0$ $\forall\mathbf{x}\in\mathbb{R}^{nN}$ and the equality holds if and only if $\mathbf{x}=\mathbf{x}^*=[x^*;\ldots;x^*]$. Hence, $V$ can be viewed as a measure of optimality error. The convergence of DEAN will be established upon the monotonicity of $V$. 

To this end, we introduce the following notations: For each $i\in\mathcal{V}$, let
\begin{align*}
\mathcal{C}_i=\{x\in \mathbb{R}^n: f_i(x^*)\!-\!f_i(x)\!-\!\nabla f_i(x)^T(x^*\!-\!x)\!\leq\! V(\mathbf{x}^0)\},
\end{align*}
where $\mathbf{x}^0=[x_1^0;\ldots;x_N^0]$ is the initial state given by \eqref{eq:xi0}. Clearly, $x_i^0,x^*\in\mathcal{C}_i$. We additionally impose the following condition, which ensures $\mathcal{C}_i$ $\forall i\in\mathcal{V}$ to be compact.
\begin{assumption}\label{asm:coercive}
For each $i\in\mathcal{V}$, $f_i(x^*)-f_i(x)-\nabla f_i(x)^T(x^*-x)$ is a coercive function of $x$.
\end{assumption}
If $f_i$ is globally strongly convex, then Assumption~\ref{asm:coercive} naturally holds. The converse is not true. For example, the scalar function equal to $x\ln x$ for all $x\ge1$ and equal to $\frac{1}{2}(x^2-1)$ for all $x<1$ satisfies all the conditions in Assumptions~\ref{f} and~\ref{asm:coercive}, but is not globally strongly convex.

Accordingly, for each $i\in\mathcal{V}$, there exist $\Theta_i,\theta_i\in(0,\infty)$ such that $\Theta_i I_n\succeq\nabla^2f_i(x)\succeq \theta_i I_n$ $\forall x\in\operatorname{conv}\mathcal{C}_i$. In other words, $\theta_i$ is the convexity parameter of $f_i$ on $\operatorname{conv}\mathcal{C}_i$ and $\Theta_i$ is the Lipschitz constant of $\nabla f_i$ on $\operatorname{conv}\mathcal{C}_i$, which guarantee 
\begin{align*}
\frac{\theta_i}{2}\|x-y\|^2&\le f_i(y)-f_i(x)-\nabla f_i(x)^T(y-x)\le\frac{\Theta_i}{2}\|x-y\|^2,\quad\forall x,y\in\operatorname{conv}\mathcal{C}_i.
\end{align*}
We let $\Theta:=\max_{i\in\mathcal{V}}\Theta_i>0$ and $\theta:=\min_{i\in\mathcal{V}}\theta_i>0$. Also, let $L_i\in[0,\infty)$ be the Lipschitz constant of $\nabla^2f_i$ on $\operatorname{conv}\mathcal{C}_i$, i.e., $\|\nabla^2f_i(x)-\nabla^2f_i(y)\|\le L_i\|x-y\|$ $\forall x,y\in\operatorname{conv}\mathcal{C}_i$. Furthermore, we let $\bar{\Theta}_i,\bar{\theta}_i\in(0,\infty)$ be such that $\bar{\Theta}_i I_n\succeq\nabla^2f_i(x)\succeq\bar{\theta}_i I_n$ $\forall x\in \operatorname{conv}\cup_{j\in\mathcal{V}}\mathcal{C}_j$. On the other hand, owing to Assumption~\ref{gijl}, there exist $\Gamma_{\{i,j\}},\gamma_{\{i,j\}}\in(0,\infty)$ such that $\Gamma_{\{i,j\}} I_n\succeq\nabla^2 g_{\{i,j\}}(x)\succeq \gamma_{\{i,j\}} I_n$ $\forall x\in\operatorname{conv}\{\mathcal{C}_i\cup\mathcal{C}_j\}$.

Now arbitrarily pick $\bar{\alpha}>0$ and suppose the step-sizes $\alpha_{\{i,j\}}$ $\forall\{i,j\}\in\mathcal{E}$ are selected from the predetermined interval $(0,\bar{\alpha}]$. Then, define
\begin{align*}
\mathcal{C}'_i=\operatorname{conv}\{x\in\mathbb{R}^n:\|x-y\|\le\delta_i,\;y\in\mathcal{C}_i\},
\end{align*}
where
\begin{align}
\delta_i\!=\!\frac{\bar{\alpha}}{\theta_i}\!\!\sum_{j\in\mathcal{N}_i}\!\Gamma_{\{i,j\}}\left(\sqrt{\frac{2V(\mathbf{x}^0)}{\theta_i}}\!+\!\sqrt{\frac{2V(\mathbf{x}^0)}{\theta_j}}\right)\!\in\![0,\infty).\label{eq:delta}
\end{align}
Since $\mathcal{C}_i$ is compact and $\delta_i\ge0$ is finite, $\mathcal{C}'_i$ is also compact. Similarly, we let $\Theta'_i,\theta'_i\in(0,\infty)$ be such that $\Theta'_i I_n\succeq\nabla^2f_i(x)\succeq\theta'_i I_n$ $\forall x\in\mathcal{C}'_i$, and let $L'_i\in[0,\infty)$ be such that $\|\nabla^2f_i(x)-\nabla^2f_i(y)\|\le L'_i\|x-y\|$ $\forall x,y\in\mathcal{C}'_i$. Observe that $\mathcal{C}'_i\supseteq\operatorname{conv}\mathcal{C}_i$, so that $\Theta'_i\ge\Theta_i$, $\theta'_i\le\theta_i$, and $L'_i\ge L_i$.

Once the local strong convexity and local Lipschitz continuity conditions of our problem become global, the above local convexity parameters and the local Lipschitz constants can be chosen as the global ones that take effect over $\mathbb{R}^n$. 

Below, our first result shows that $V(\mathbf{x}^k)$ is non-increasing in $k$ and quantifies its drop at each iteration.

\begin{lemma}[Monotonicity of Lyapunov function]\label{lem:monoto} 
Suppose Assumptions~\ref{f},~\ref{gijl}, and~\ref{asm:coercive} hold. Let $\mathbf{x}^k=[x_1^k;\ldots;x_N^k]$ $\forall k\ge0$ be generated by DEAN described in Algorithm~\ref{alg:DEAN} with $0<\alpha_{\{i,j\}}\le\bar{\alpha}$ $\forall\{i,j\}\in\mathcal{E}$. If, in addition,
	\begin{align}\label{eq:stepsizedrop}
	&\alpha_{\{i,j\}}<\frac{1}{2\Gamma_{\{i,j\}}}\min\Bigl\{\frac{\theta_i^2}{|\mathcal{N}_i|}\left(\Theta'_i-\frac{\theta'_i}{2}+\frac{L'_i}{2}\sqrt{\frac{2V(\mathbf{x}^0)}{\theta_i}}
	\right)^{-1},\frac{\theta_j^2}{|\mathcal{N}_j|}\left(\Theta'_j-\frac{\theta'_j}{2}+\frac{L'_j}{2}\sqrt{\frac{2V(\mathbf{x}^0)}{\theta_j}}
	\right)^{-1}\Bigr\},\;\forall\{i,j\}\in\mathcal{E},
	\end{align}
	then for each $k\ge0$,
	\begin{align}\label{eq:Vdrop}
	&V(\mathbf{x}^{k+1})\!-\!V(\mathbf{x}^k)\le-\!\sum_{i\in\mathcal{V}}\sum_{j\in\mathcal{N}_i}\!\alpha_{\{i,j\}}\gamma_{\{i,j\}}^2\|x_j^k-x_i^k\|^2\left[\left(\frac{\theta_i}{2}-\Theta_i-\frac{L_i}{2}\sqrt{\frac{2V(\mathbf{x}^0)}{\theta_i}}\right)\frac{|\mathcal{N}_i|}{\theta_i^2}\alpha_{\{i,j\}}+\frac{1}{2\Gamma_{\{i,j\}}}\right]\le0.
	\end{align}
\end{lemma}

\begin{proof}
	See Appendix~\ref{sec:proofofstepsize}.
\end{proof}

From Lemma~\ref{lem:monoto}, $V$ is a Lyapunov function which keeps strictly decreasing until $x_i^k$ $\forall i\in\mathcal{V}$ become identical. Also, since $V(\mathbf{x}^k)$ is bounded from below, $\lim_{k\rightarrow\infty}V(\mathbf{x}^k)$ exists (but is possibly nonzero). 

\begin{remark}
In Lemma~\ref{lem:monoto} as well as the statements in the rest of the paper, the constant $\bar{\alpha}$ in the condition $0<\alpha_{\{i,j\}}\le\bar{\alpha}$ $\forall\{i,j\}\in\mathcal{E}$ is an arbitrarily given positive scalar, which determines $\delta_i$ and thus the compact set $\mathcal{C}'_i$. Specifically, the larger $\bar{\alpha}$ is, the larger $\mathcal{C}'_i$ is. Accordingly, $\delta_i$ also affects the values of $\theta'_i$, $\Theta'_i$, and $L'_i$.
\end{remark}

Lemma~\ref{lem:monoto} leads to the next result, which says that the entire network of nodes eventually reach a consensus.

\begin{theorem}[Asymptotic convergence to consensus]\label{theo:firsttheo} 
Suppose Assumptions~\ref{f},~\ref{gijl}, and~\ref{asm:coercive} hold. Let $\mathbf{x}^k=[x_1^k;\ldots;x_N^k]$ $\forall k\ge0$ be generated by DEAN described in Algorithm~\ref{alg:DEAN} with $0<\alpha_{\{i,j\}}\le\bar{\alpha}$ $\forall\{i,j\}\in\mathcal{E}$. Suppose \eqref{eq:stepsizedrop} holds. Then,
\begin{align}
\lim_{k\rightarrow \infty} \|x_i^k-x_j^k\| = 0,\quad\forall i,j\in \mathcal{V}.\label{eq:consensuserrorbound}
\end{align} 
\end{theorem}

\begin{proof}
See Appendix~\ref{sec:proofoftheo:firsttheo}.
\end{proof}

In the following theorem, we show the nodes are able to reach $\epsilon$-accuracy in optimality for any given $\epsilon>0$, provided that the step-sizes $\alpha_{\{i,j\}}$ $\forall\{i,j\}\in\mathcal{E}$ are properly related to $\epsilon$.

\begin{theorem}[Arbitrary closeness to exact optimality]\label{thm:convepsilon}	
Suppose Assumptions~\ref{f},~\ref{gijl}, and~\ref{asm:coercive} hold. Let $\mathbf{x}^k\!\!=\![x_1^k;\ldots;x_N^k]$ $\!\forall k\!\ge\!0$ be generated by DEAN described in Algorithm~\ref{alg:DEAN} with $0<\alpha_{\{i,j\}}\le\bar{\alpha}$ $\forall\{i,j\}\in\mathcal{E}$. For each $i\in\mathcal{V}$, let $\eta_i=|\mathcal{N}_i|L_i\sqrt{N}V(\mathbf{x}^0)/(\theta_i^2\sum_{\ell\in \mathcal{V}} \bar{\theta}_\ell)\ge0$ and $\tilde{\eta}_i = 2|\mathcal{N}_i|\left(\Theta'_i-\frac{\theta'_i}{2}+\frac{L'_i}{2}\sqrt{\frac{2V(\mathbf{x}^0)}{\theta_i}}\right)/\theta_i^2>0$. Given $\epsilon>0$, if
\begin{align}
\!\!\!\alpha_{\{i,j\}}<\frac{\epsilon}{\Gamma_{\{i,j\}}}\min\Bigl\{\frac{1}{\eta_i+\tilde{\eta}_i\epsilon},\frac{1}{\eta_j+\tilde{\eta}_j\epsilon}\Bigr\},\;\forall\{i,j\}\in\mathcal{E},\label{eq:stepsizeepsilon}
	\end{align}
then $\lim_{k\rightarrow\infty}\|\mathbf{x}^k-\mathbf{x}^*\|<\epsilon$.
\end{theorem}
	
\begin{proof}
See Appendix~\ref{sec:proofofthm:convergence}.
\end{proof}

Theorem~\ref{thm:convepsilon}	 says that to achieve a suboptimal solution of given accuracy $\epsilon>0$, the local step-sizes $\alpha_{\{i,j\}}$'s should depend on $\epsilon$ as in \eqref{eq:stepsizeepsilon}. Observe that the $\alpha_{\{i,j\}}$'s may be very small if $\epsilon$ is tiny, which might result in slow convergence. There appears to be an inevitable trade-off between convergence speed and accuracy, and this is indeed very common for optimization methods. Moreover, the step-size condition is for the purpose of establishing theoretical results, and larger step-sizes that are empirically selected and may violate the theoretical condition are often adopted in practice.

Subsequently, we investigate the convergence rate of DEAN. For simplicity, here we only consider DEAN with $g_{\{i,j\}}(x)=\frac{1}{2}x^Tx$ $\forall \{i,j\}\in\mathcal{E}$. In fact, we can always allow more general $g_{\{i,j\}}$'s under Assumption~\ref{gijl}, which would lead to more complicated but same-order convergence rates. We show in the theorem below that the nodes achieve a consensus with an explicit error bound at a \emph{linear rate} that depends on $\lambda_{\max}(L_{\mathcal{G}})$ and $\lambda_{2}(L_{\mathcal{G}})$. Here, $L_{\mathcal{G}}$ is the Laplacian matrix of the graph $\mathcal{G}$, which is symmetric positive semidefinite and equal to $\hat{H}_\mathcal{G}$ in \eqref{matrix:H_G} with $h_{\{i,j\}}=1$ $\forall\{i,j\}\in\mathcal{E}$. Since $\mathcal{G}$ is connected, $\lambda_{2}(L_{\mathcal{G}})$, which is called the algebraic connectivity of $\mathcal{G}$, is positive. Also, $0<\lambda_{\max}(L_{\mathcal{G}})\le\min\{N,2\max_{i\in\mathcal{V}}|\mathcal{N}_i|\}$. 

\begin{theorem}[Linear convergence rate]\label{thm:rate}
Suppose Assumptions~\ref{f} and~\ref{asm:coercive} hold. Let $\mathbf{x}^k\!=\![x_1^k;\ldots;x_N^k]$ $\forall k\ge0$ be generated by DEAN described in Algorithm~\ref{alg:DEAN} with $0<\alpha_{\{i,j\}}\le\bar{\alpha}$ and $g_{\{i,j\}}(x)=\frac{1}{2}x^Tx$ $\forall\{i,j\}\in\mathcal{E}$. Suppose \eqref{eq:stepsizedrop} holds and $\alpha_{\{i,j\}}<\theta/(\max_{i\in\mathcal{V}}|\mathcal{N}_i|)$ $\forall\{i,j\}\in\mathcal{E}$. Then, there exists $\tilde{\mathbf{x}}=[\tilde{x};\ldots;\tilde{x}]\in\mathbb{R}^{nN}$, $\tilde{x}\in\mathbb{R}^n$ such that
\begin{align}
\|\mathbf{x}^k-\tilde{\mathbf{x}}\|\leq\dfrac{\max_{\{i,j\}\in \mathcal{E}}\alpha_{\{i,j\}}\lambda_{\max}(L_{\mathcal{G}})}{\theta(1-q)}\|\mathbf{x}^0\|q^k,\label{eq:xx<=Cqk}
\end{align}
where $q=\max\big\{\max_{\{i,j\}\in\mathcal{E}}\alpha_{\{i,j\}}\lambda_{\max}(L_\mathcal{G})/\theta-1, 1-\min_{\{i,j\}\in\mathcal{E}}\alpha_{\{i,j\}}\lambda_2(L_{\mathcal{G}})/\Theta\big\}\in(0,1)$. In addition,
\begin{align}
&\|\tilde{\mathbf{x}}-\mathbf{x}^*\|\leq\max_{i\in\mathcal{V}}\dfrac{L_i}{\tilde{\rho}_i}\cdot\dfrac{\sqrt{N}V(\mathbf{x}^0)}{2\sum_{i\in\mathcal{V}}\bar{\theta}_i},\label{eq:optimalityerror}
\end{align}
where $\tilde{\rho}_i=\frac{\theta_i^2}{2|\mathcal{N}_i|\!\max_{j\in\mathcal{N}_i}\alpha_{\{i,j\}}}+\frac{\theta_i}{2}-\Theta_i-\frac{L_i}{2}\sqrt{\frac{2V(\mathbf{x}^0)}{\theta_i}}>0$ $\forall i\in\mathcal{V}$.
\end{theorem}

\begin{proof}
See Appendix~\ref{sec:proofofthm:consensuserror}.
\end{proof}

The linear rate $q$ of convergence in Theorem~\ref{thm:rate} depends on the curvatures of the objective functions characterized by $\theta$ and $\Theta$, the network connectivity characterized by $\lambda_2(L_{\mathcal{G}})$ and $\lambda_{\max}(L_\mathcal{G})$, and the algorithm parameters $\alpha_{\{i,j\}}$ $\forall\{i,j\}\in\mathcal{E}$. In addition, a bound on the distance between the consensus $\tilde{x}$ asymptotically reached by all the nodes and the optimum $x^*$ of problem~\eqref{pro} is provided in Theorem~\ref{thm:rate}. It follows from \eqref{eq:optimalityerror} that $\|\tilde{x}-x^*\|\le O(\max_{\{i,j\}\in\mathcal{E}}\alpha_{\{i,j\}})$. This suggests that the smaller the $\alpha_{\{i,j\}}$'s are, the closer $\tilde{x}$ is to $x^*$, which is consistent with Theorem~\ref{thm:convepsilon}. Furthermore, as is discussed in Section~\ref{sec:AZGS}, if each $f_i$ is given by \eqref{eq:quadratic}, we have $\sum_{i\in\mathcal{V}}\nabla f_i(x_i^k)\equiv 0$ and therefore each $x_i^k$ is guaranteed to linearly converge to the exact optimum $x^*$.

\subsection{Comparison with Existing Linearly Convergent Methods} \label{ssec:convcomp}

\begin{table*}[tb]
  \centering
    \begin{tabular}{|c|c|c|c|c|c|}
      \hline
			Algorithm & Type & Strong convexity of $f_i$  & Lipschitz continuity of $\nabla f_i$  & Lipschitz continuity of $\nabla^2 f_i$ & Solution\\
      \hline
      DEAN & second-order & local & local & local & inexact\\
      \hline
			NN \cite{Mokhtari17} & second-order & global & global & global & inexact\\
      \hline
			ANN \cite{Mansoori19} & second-order & global & global & global & inexact\\
      \hline
			DQN \cite{Bajovic15} & second-order & global & global & -- & inexact \\
      \hline
		  D-BFGS \cite{Eisen17} & quasi-Newton & global & global & -- & inexact \\
      \hline
			ESOM \cite{Mokhtari16A} & second-order & global & global & global & exact\\
      \hline
			DQM \cite{Mokhtari16B} & second-order & global & global & global & exact\\
      \hline
			PD-QN \cite{Eisen19} & quasi-Newton & global & global & -- & exact\\
      \hline
			SDD-based Newton \cite{Tutunov19} & second-order & global & global & global for $(\nabla^2 f_i)^{-1}$ & exact\\
      \hline
			NRC \cite{Varagnolo16} & second-order & global & global & local & exact\\
      \hline
			DGD \cite{YuanK16} & first-order & global & global & -- & inexact\\
      \hline
			EXTRA \cite{ShiW15} & first-order & global & global & -- & exact\\
      \hline
			DIGing \cite{Nedic17} & first-order & global & global & -- & exact\\
      \hline
    \end{tabular}
		\vspace*{0.1in}
  \caption{{\upshape Comparison of Linearly Convergent Methods.}} \label{table:comparison}
\end{table*}

In Table~I, we tabulate the key assumptions required by DEAN and a number of existing methods for deriving their linear convergence rates. The existing linearly convergent methods involved for comparison here include the prior decentralized second-order methods \cite{Mokhtari16A,Mokhtari16B,Eisen17,Mokhtari17,Bajovic15,Varagnolo16,Eisen19,Tutunov19,Mansoori19} and several typical distributed first-order methods \cite{YuanK16,ShiW15,Nedic17}, among which the quasi-Newton methods \cite{Eisen17,Eisen19} and the first-order methods \cite{YuanK16,ShiW15,Nedic17} do not utilize Hessian matrices to update. Like DEAN, NN \cite{Mokhtari17}, ANN \cite{Mansoori19}, DQN \cite{Bajovic15}, D-BFGS \cite{Eisen17}, and DGD \cite{YuanK16} converge to an inexact solution, yet \cite{Mokhtari17,Mansoori19,Bajovic15,Eisen17} do not offer explicit error bounds on their inexact solutions as Theorem~\ref{thm:rate} does. ESOM \cite{Mokhtari16A}, DQM \cite{Mokhtari16B}, PD-QN \cite{Eisen19}, SDD-based Newton \cite{Tutunov19}, NRC \cite{Varagnolo16}, EXTRA \cite{ShiW15}, and DIGing \cite{Nedic17} are able to converge to the exact optimum.

The linear convergence of DEAN is established under less restrictive or different assumptions compared to the aforementioned methods. Note that $f_i$ under Assumptions~\ref{f} and~\ref{asm:coercive} is allowed to be locally strongly convex and with locally Lipschitz continuous $\nabla f_i$ and $\nabla^2 f_i$. In contrast, the existing methods in Table~I assume $f_i$ to be globally strongly convex and $\nabla f_i$ to be globally Lipschitz continuous when deriving their linear convergence rates. Moreover, our assumption on the local Lipschitz continuity of $\nabla^2f_i$ is weaker than the three times continuous differentiability of $f_i$ in \cite{Varagnolo16} as well as the global Lipschitz continuity of $\nabla^2f_i$ in \cite{Mokhtari16A,Mokhtari16B,Mokhtari17,Mansoori19}, is different from the global Lipschitz continuity of $(\nabla^2 f_i)^{-1}$ in \cite{Tutunov19}, yet is not needed in \cite{Bajovic15,Eisen17,Eisen19,YuanK16,ShiW15,Nedic17}.

Among the existing algorithms in Table~I, the second-order methods ANN, DQN, and D-BFGS and the first-order methods DGD, EXTRA, and DIGing indeed do not require global strong convexity to make all the iterates bounded, which suggests that the convergence analyses in \cite{Mansoori19,Bajovic15,Eisen17,YuanK16,ShiW15,Nedic17} may be extended to locally strongly convex problems. Nevertheless, the parameter conditions of these methods need the global Lipschitz constant of $\nabla f_i$ to be finite, which plays a role in guaranteeing the boundedness of the iterates. The remaining existing algorithms in Table~I all require the global strong convexity of $f_i$ and the global Lipschitz continuity of $\nabla f_i$, without which their parameter conditions would be invalid and their analyses would not work. Therefore, it is nontrivial to extend the analyses of the existing methods in Table~I to problems under our assumptions. 

\section{Numerical Examples} \label{sec:numeexam}

In this section, we demonstrate the competent performance of DEAN via comparisons with a number of typical decentralized second-order and first-order methods.

Consider the following logistic regression problem \cite{Bach14} that often arises in machine learning: Suppose there are $N$ interconnected nodes, and $m$ training samples $(u_{ij}, v_{ij})\in \mathbb{R}^n\times \{-1,+1\}$ $\forall j=1,\dots,m$ are assigned to each node $i$, where $u_{ij}\in\mathbb{R}^n$ is the feature vector whose last element is $1$ and $v_{ij}\in\{-1,+1\}$ is the label. The goal is to predict the probability $P(v =1|u)=1/(1+\text{exp}(-u^Tx))$ of having label $v = 1$ given a feature vector $u$ whose class is unknown. This is equivalent to
\begin{equation}\label{pro:log}
\min_{x\in \mathbb{R}^n} \sum_{i=1}^N \sum_{j=1}^{m}\ln(1+\exp(-v_{ij}u_{ij}^Tx)).
\end{equation}
Problem~\eqref{pro:log} is locally strongly convex provided that the sample number $m$ is sufficiently large. In the simulations, we consider both low-dimensional and high-dimensional cases with $n=10$ and $n=30$, respectively, and set $N=20$ and $m=2n$. We suppose each node has $n$ samples with label $1$ and $n$ samples with label $-1$. The first $n-1$ elements of each feature vector $u_{ij}$ are generated from normal distributions with variance $1$ and with mean $\pm10$ corresponding to $v_{ij}=\pm1$, respectively. We also consider networks of different densities by choosing the average degree (i.e., the average of $|\mathcal{N}_i|$ $\forall i\in\mathcal{V}$) to be $4$ and $10$. To generate such networks, we first generate a connected $(N-1)$-link graph and then randomly add links until the average degree reaches the given value.

We compare DEAN with the state-of-the-art second-order methods DQM \cite{Mokhtari16B}, DQN-K \cite{Bajovic15}, ESOM-K \cite{Mokhtari16A}, NN-K \cite{Mokhtari17}, and NRC \cite{Varagnolo16}, as well as two recent first-order methods EXTRA \cite{ShiW15} and DIGing \cite{Nedic17}. For NN-K, DQN-K, and ESOM-K, we only consider $\text{K}=0$, in which case their communication costs per iteration are the same as that of DEAN. 

In the simulation results that we present below, the algorithm parameters are set as follows: The weight matrix $W$ in NN and DQN is set as $[W]_{ij}=1/(\max\{|\mathcal{N}_i|,|\mathcal{N}_j|\}+2)$ $\forall\{i,j\}\in \mathcal{E}$, $[W]_{ij}=0$ $\forall \{i,j\}\notin\mathcal{E}$, $i\neq j$, and $[W]_{ii}=1-\sum_{j\in \mathcal{N}_i}[W]_{ij}$ $\forall i\in\mathcal{V}$. In addition, the implementation of NRC needs a positive global convexity parameter of $\sum_{i\in\mathcal{V}}f_i(x)$, while problem~\eqref{pro:log} does not have that. Thus, we artificially set the convexity parameter in NRC to be $1$. The remaining parameters are all hand-optimized, in the sense that each of them is picked from a great number of discrete points spanning a sufficiently large interval (which strictly includes the range for convergence in practice) such that the resulting optimality error after a given number $T$ of iterations is minimal. Here, the optimality error is defined as $e(k)=\|\sum_{i\in\mathcal{V}}\nabla f_i(x_i^k)\|+(\sum_{i\in\mathcal{V}}\|x_i^k-\frac{1}{N}\sum_{j\in\mathcal{V}}x_j^k\|^2)^{\frac{1}{2}}$, which involves the global objective gradient and the disagreements among the nodes. Moreover, the initialization \eqref{eq:xi0} of DEAN is calculated using YALMIP \cite{Lofberg04}. For fair comparison, we let all the algorithms start from the same initial states. 

Figures~\ref{fig:local_dim10} and~\ref{fig:local_dim30} illustrate the practical convergence performance during $T=100$ iterations for each of the aforementioned algorithms when solving $10$-dimensional and $30$-dimensional logistic regression problems on networks with average degrees $4$ and $10$, respectively. To reveal the convergence speed, we plot $e(k)/e(0)$ in the figures. Observe from Figures~\ref{fig:local_dim10} and~\ref{fig:local_dim30} that lower problem dimension and denser network connectivity lead to faster convergence of DEAN. Compared to the aforementioned existing algorithms, initially DEAN has no advantage in error reduction, yet it catches up quickly and eventually outperforms the other methods. 

\begin{figure*}[!htb]
\centering
\subfigure[Network of average degree $4$]{
\includegraphics[width=0.35\linewidth, height=0.27\linewidth]{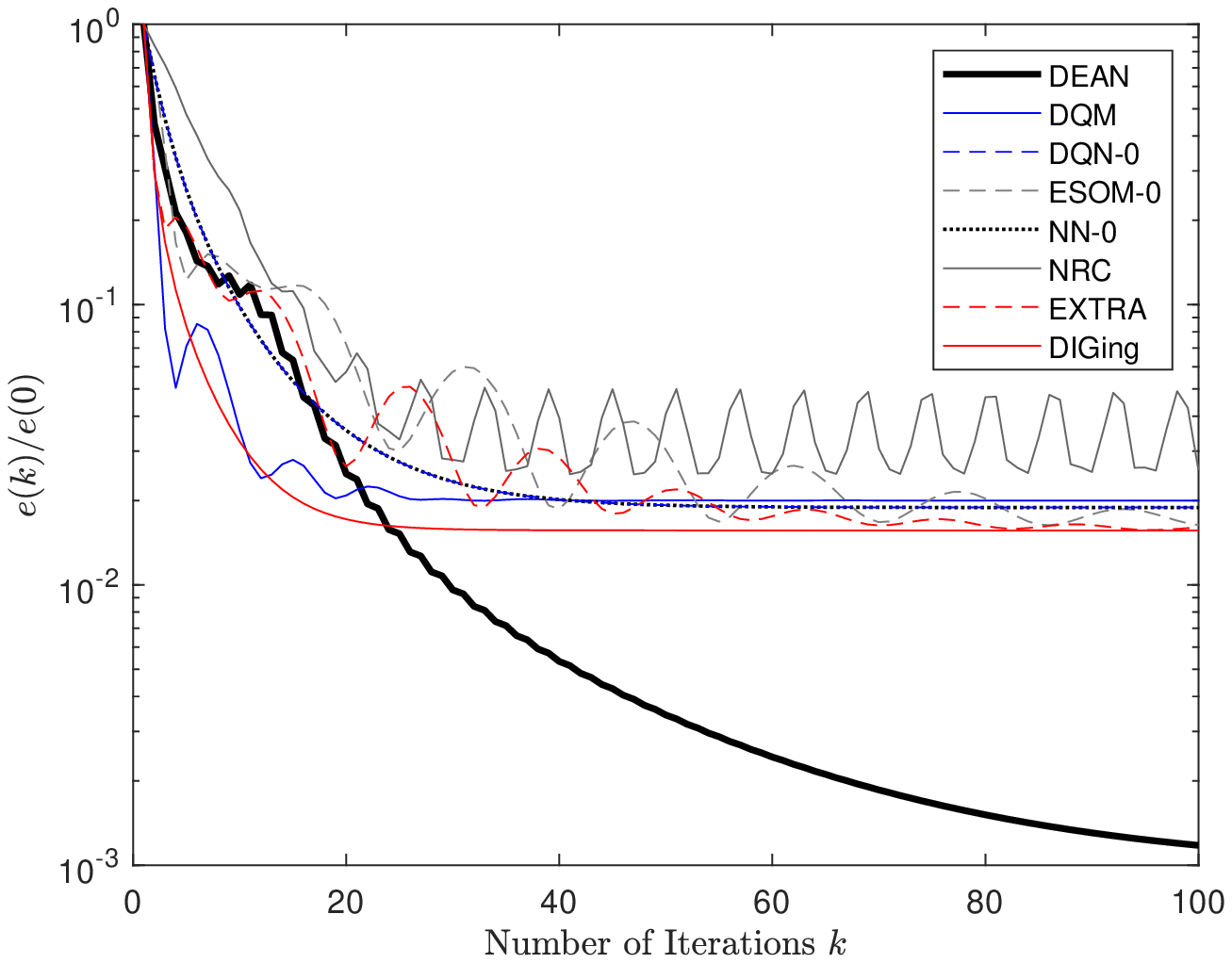}\label{fig:(10a)}}
\hspace{0.1\linewidth}
\subfigure[Network of average degree $10$]{
\includegraphics[width=0.35\linewidth, height=0.27\linewidth]{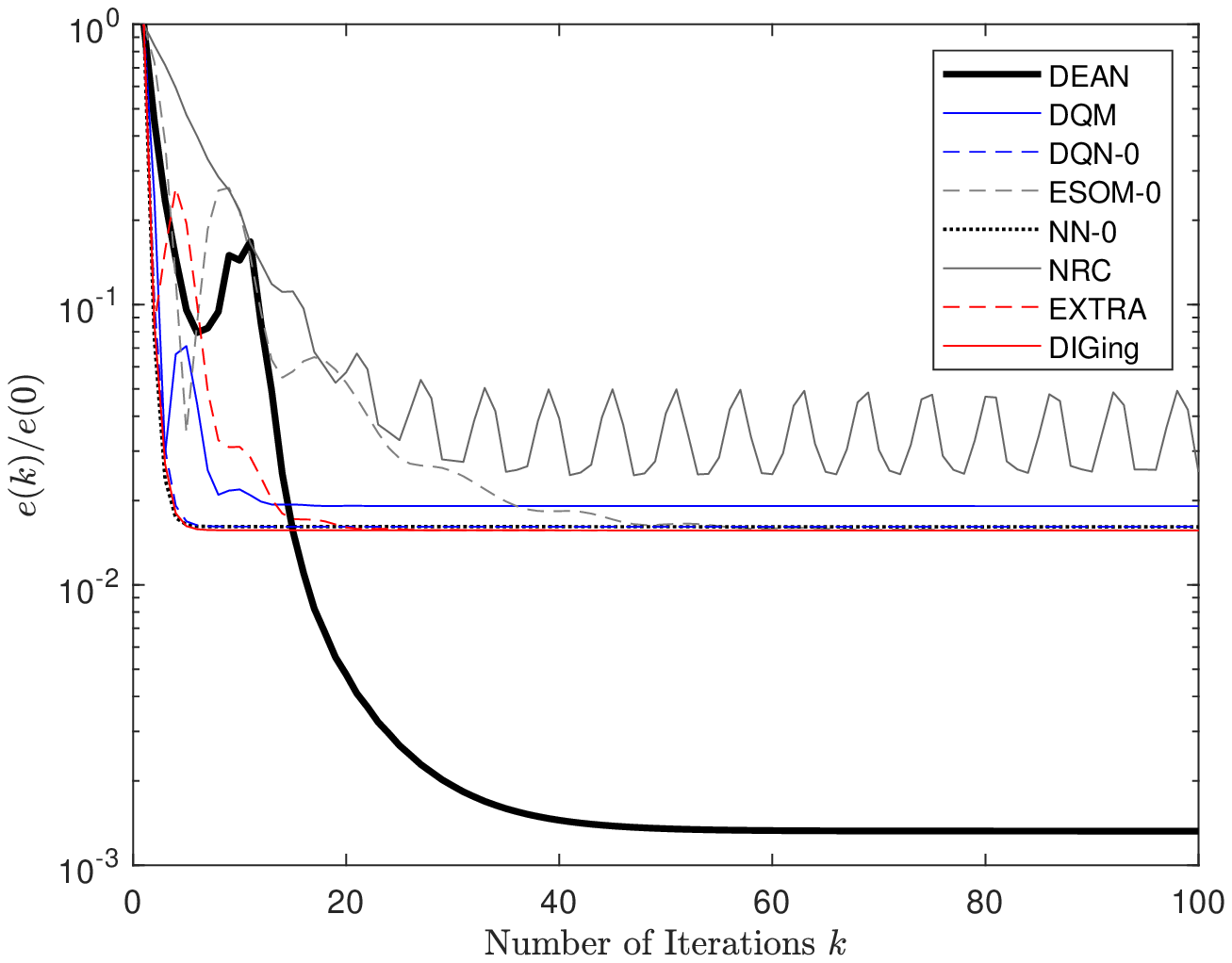}\label{fig:(10b)}}
\caption{Convergence performance for $10$-dimensional problem.}
\label{fig:local_dim10}
\end{figure*}

\begin{figure*}[!htb]
\centering
\subfigure[Network of average degree $4$]{
\includegraphics[width=0.35\linewidth, height=0.27\linewidth]{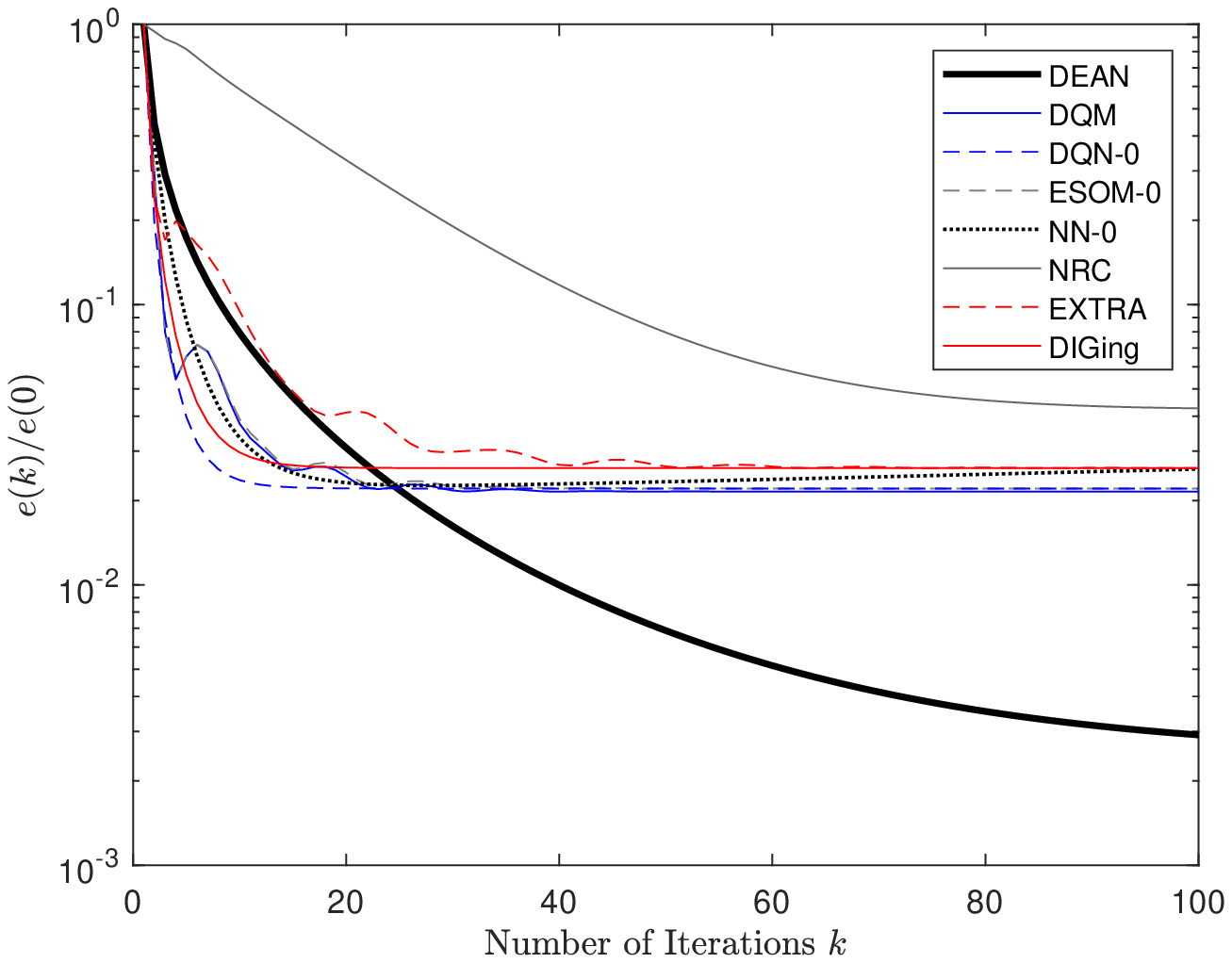}\label{fig:(30a)}}
\hspace{0.1\linewidth}
\subfigure[Network of average degree $10$]{
\includegraphics[width=0.35\linewidth, height=0.27\linewidth]{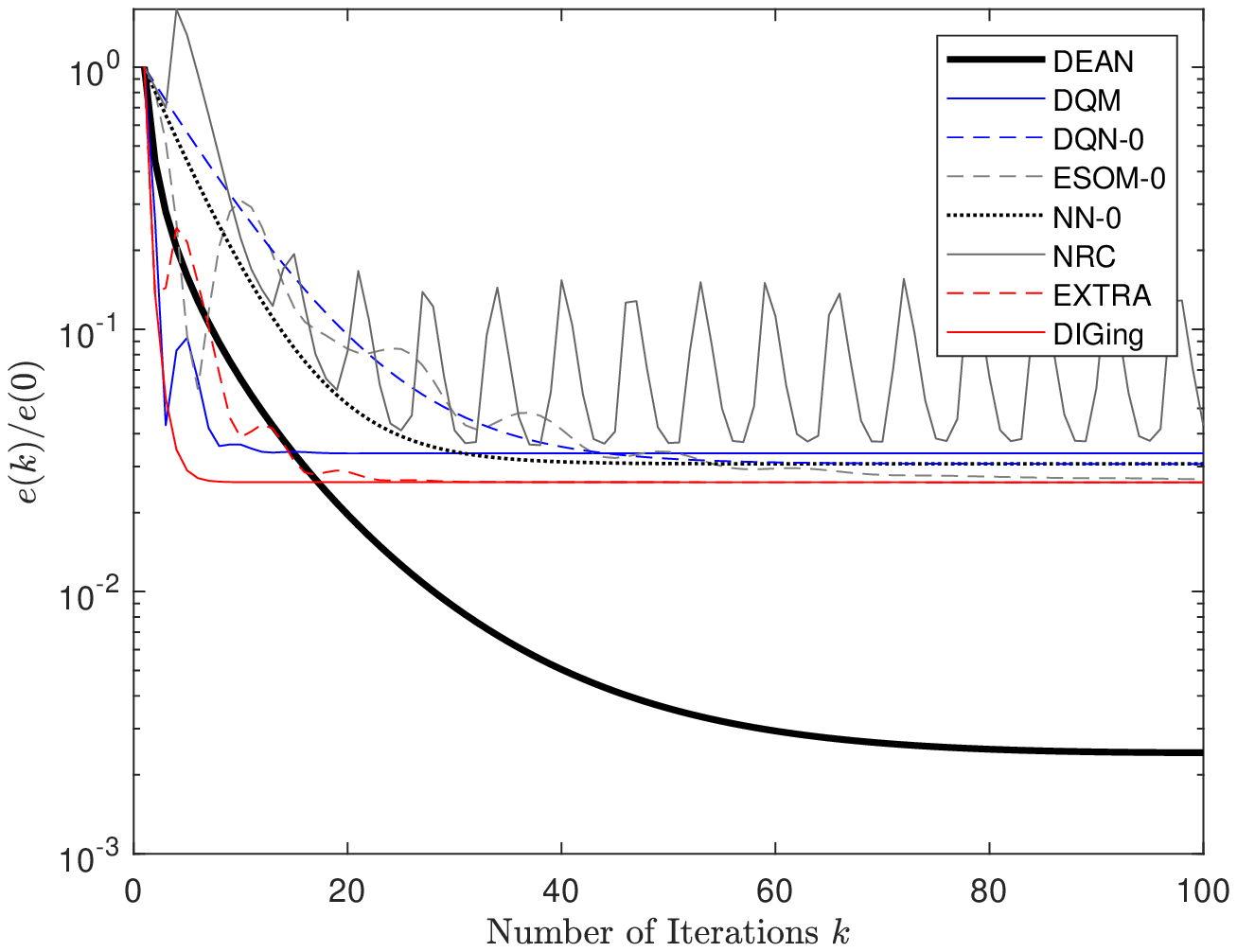}\label{fig:(30b)}}
\caption{Convergence performance for $30$-dimensional problem.}
\label{fig:local_dim30}
\end{figure*}

For the four scenarios in Figures~\ref{fig:local_dim10} and~\ref{fig:local_dim30}, Tables~II and~III list the corresponding running time for each algorithm to complete $T=100$ iterations via MATLAB on a standard computer (Intel(R) Core(TM) i7-8565U CPU @1.80GHz). It can be seen that DEAN takes less time than the second-order methods DQM, DQN-0, ESOM-0, NN-0, and NRC in all of the four scenarios, which indicates that DEAN is a computationally efficient second-order algorithm. Compared to the first-order methods EXTRA and DIGing, DEAN mostly requires less time when the problem dimension is $10$ but needs more time when the problem dimension increases to $30$. This suggests that the first-order methods are more scalable with respect to the problem dimension, which is natural because the first-order methods do not need to evaluate Hessian inverses as the second-order methods do.

\begin{table}[ht]
\centering
	\begin{tabular}{|l|c|c|}
		\hline
Algorithm & Average degree $4$ & Average degree $10$ \\ \hline
DEAN& 179.4&183.2\\
\hline
DQM&323.6&346.3\\
\hline
DQN-0&271.5&303.7\\
\hline
ESOM-0&317.6&374.7\\
\hline
NN-0&318.6&307.5\\
\hline
NRC&1400.6&3154.6\\
\hline
EXTRA&164.1&222.9\\
\hline
DIGing&211.1&221.2\\
\hline
\end{tabular}
\vspace{0.03in}
\caption{Running time (\textit{ms}) of $100$ iterations for $10$-dimensional problem.}
\end{table}

\begin{table}[ht]
\centering
	\begin{tabular}{|l|c|c|}
		\hline
Algorithm & Average degree $4$ & Average degree $10$ \\ \hline
DEAN& 501.3 &  497.8\\
\hline
DQM&869.5 & 783.5 \\
\hline
DQN-0&809.4 & 917.8\\
\hline
ESOM-0&998.7&895.7\\
\hline
NN-0&914.1&960.9\\
\hline
NRC&4434.6&9180.9\\
\hline
EXTRA&366.8&394.3\\
\hline
DIGing&393.8&390.0\\
\hline
\end{tabular}
\vspace{0.03in}
\caption{Running time (\textit{ms}) of $100$ iterations for $30$-dimensional problem.}
\end{table}

\section{Conclusion}\label{sec:conclusion}

In this paper, we have developed a set of novel decentralized approximate Newton (DEAN) methods for convex optimization on networked systems. With appropriate algorithm parameters, the DEAN algorithms allow nodes in a network to reach a consensus at a linear rate, which can be sufficiently close to the optimal solution. Compared with most existing decentralized second-order methods, the DEAN algorithms relax the global strong convexity assumption on the objective functions to local strong convexity and still establish the convergence. Finally, we have demonstrated the competitive convergence performance as well as the computational efficiency of DEAN via numerical examples. Possible extensions of this work include developing broadcast versions of DEAN, allowing block updates of the local states, and considering packet losses in communications.



\appendix
\subsection{Proof of Lemma~\ref{lem:monoto}}\label{sec:proofofstepsize}

For each $k\ge0$, define for convenience that $\Delta V(\mathbf{x}^k)=V(\mathbf{x}^{k+1})-V(\mathbf{x}^{k})$ and
\begin{equation}\label{phiijk1}
\phi_{ij}^{k}=\nabla g_{\{i,j\}}(x_j^{k})-\nabla g_{\{i,j\}}(x_i^{k}),\quad\forall\{i,j\}\in\mathcal{E}.
\end{equation}
Note that $\phi_{ij}^k=-\phi_{ji}^k$. 

Below, we prove by induction that $V(\mathbf{x}^k)\le V(\mathbf{x}^0)$ $\forall k\ge0$. Clearly, this is true for $k=0$. Now suppose $V(\mathbf{x}^\ell)\le V(\mathbf{x}^0)$ for some $\ell\ge 0$, and we want to prove $V(\mathbf{x}^{\ell+1})\le V(\mathbf{x}^0)$ by showing that $\Delta V(\mathbf{x}^\ell)\le 0$.

Toward showing $\Delta V(\mathbf{x}^\ell)\le 0$, we first show that $x_i^{\ell+1}\in\mathcal{C}'_i$ $\forall i\in\mathcal{V}$. Note from $V(\mathbf{x}^\ell)\le V(\mathbf{x}^0)$ that $x_i^\ell\in\mathcal{C}_i$. Thus, due to the local strong convexity of $f_i$,
\begin{align}
&\|x_i^\ell-x^*\|\le\sqrt{\frac{2V(\mathbf{x}^0)}{\theta_i}},\label{eq:xx*<=sqrt2Vtheta}\displaybreak[0]\\
&\|(\nabla^2f_i(x_i^\ell))^{-1}\|\le\frac{1}{\theta_i}.\label{eq:nabla2f<=1theta}
\end{align}
Because of \eqref{scheme}, \eqref{eq:nabla2f<=1theta}, \eqref{phiijk1}, and \eqref{eq:xx*<=sqrt2Vtheta}, and because $\nabla g_{\{i,j\}}$ is Lipschitz on $\operatorname{conv}\{\mathcal{C}_i\cup\mathcal{C}_j\}$ which contains both $x_i^\ell$ and $x_j^\ell$,
\begin{align}
\|x_i^{\ell+1}-x_i^\ell\|&=\|(\nabla^2f_i(x_i^\ell))^{-1}\sum_{j\in\mathcal{N}_i}\alpha_{\{i,j\}}\phi_{ij}^\ell\|\nonumber\displaybreak[0]\\
&\le\frac{1}{\theta_i}\sum_{j\in\mathcal{N}_i}\alpha_{\{i,j\}}\|\phi_{ij}^\ell\|\label{eq:xx<=1thetaalphaphi}\displaybreak[0]\\
&\le\frac{1}{\theta_i}\sum_{j\in\mathcal{N}_i}\alpha_{\{i,j\}}\Gamma_{\{i,j\}}\|x_i^\ell-x_j^\ell\|\nonumber\displaybreak[0]\\
&\le\frac{\bar{\alpha}}{\theta_i}\sum_{j\in\mathcal{N}_i}\Gamma_{\{i,j\}}(\|x_i^\ell-x^*\|+\|x_j^\ell-x^*\|)\nonumber\displaybreak[0]\\
&\le\frac{\bar{\alpha}}{\theta_i}\sum_{j\in\mathcal{N}_i}\Gamma_{\{i,j\}}\left(\sqrt{\frac{2V(\mathbf{x}^0)}{\theta_i}}+\sqrt{\frac{2V(\mathbf{x}^0)}{\theta_j}}\right).\nonumber
\end{align}
Since the right-hand side of the above inequality is exactly $\delta_i$ given in \eqref{eq:delta} and since $x_i^\ell\in\mathcal{C}_i$, we have $x_i^{\ell+1}\in\mathcal{C}'_i\supseteq\operatorname{conv}\mathcal{C}_i$. 

Next, we provide an upper bound on $\Delta V(\mathbf{x}^\ell)$. Note that $\Delta V(\mathbf{x}^\ell)$ can be written as
\begin{align}
\Delta V(\mathbf{x}^\ell)=&\sum_{i\in \mathcal{V}}\left(f_i(x_i^{\ell})-f_i(x_{i}^{\ell+1})+\nabla f_i(x_{i}^{\ell})^T(x_{i}^{\ell+1}-x_{i}^{\ell})\right) \displaybreak[0] \nonumber\\
&+\sum_{i\in \mathcal{V}}\left(\nabla f_i(x_{i}^{\ell+1})-\nabla f_i(x_{i}^{\ell})\right)^T(x_{i}^{\ell+1}-x_{i}^{\ell})+\sum_{i\in \mathcal{V}}(x_{i}^{\ell}-x^*)^T\left(\nabla f_i(x_{i}^{\ell+1})-\nabla f_i(x_{i}^{\ell})\right).\label{eq:deltaV}
\end{align}
Applying $\nabla f_i(x_i^{\ell+1})-\nabla f_i(x_i^{\ell})=\int_{0}^1\nabla^2f_i(x_i^{\ell}+s(x_i^{\ell+1}-x_i^{\ell}))(x_i^{\ell+1}-x_i^{\ell})ds$ to the term $\sum_{i\in \mathcal{V}}(x_{i}^{\ell}-x^*)^T\left(\nabla f_i(x_{i}^{\ell+1})-\nabla f_i(x_{i}^{\ell})\right)$ in \eqref{eq:deltaV} yields
\begin{align}
&\sum_{i\in \mathcal{V}}(x_i^{\ell}-x^*)^T(\nabla f_i(x_i^{\ell+1})-\nabla f_i(x_i^{\ell}))\displaybreak[0]\nonumber\\
=&\sum_{i\in \mathcal{V}}(x_i^{\ell}-x^*)^T\nabla^2f_i(x_i^{\ell})(x_i^{\ell+1}-x_i^{\ell})\displaybreak[0]\nonumber\\
+&\sum_{i\in \mathcal{V}}(x_i^{\ell}\!-\!x^*)^T\!\!\int_{0}^1\!\!(\nabla^2 f_i(x_i^{\ell}\!+\!s(x_i^{\ell+1}\!-\!x_i^{\ell}))\!-\!\nabla^2 f_i(x_i^{\ell}))ds\cdot(x_i^{\ell+1}-x_i^{\ell}).\label{xisi}
\end{align}
On the right-hand side of \eqref{xisi}, the term $\sum_{i\in \mathcal{V}}(x_i^{\ell}-x^*)^T\nabla^2f_i(x_i^{\ell})(x_i^{\ell+1}-x_i^{\ell})$ can be written as $\sum_{i\in \mathcal{V}}(x_i^{\ell}-x^*)^T\sum_{j\in \mathcal{N}_i}\alpha_{\{i,j\}}\phi_{ij}^{\ell}$ due to \eqref{scheme} and \eqref{phiijk1}. Moreover, because $\phi_{ij}^\ell=-\phi_{ji}^\ell$ and $\alpha_{\{i,j\}}=\alpha_{\{j,i\}}$ $\forall i\in\mathcal{V}$ $\forall j\in\mathcal{N}_i$, we have $\sum\limits_{i\in\mathcal{V}}\sum\limits_{j\in\mathcal{N}_i}\alpha_{\{i,j\}}\phi_{ij}^\ell=0$. Therefore,
\begin{align}
&\sum_{i\in \mathcal{V}}(x_i^{\ell}-x^*)^T\nabla^2f_i(x_i^{\ell})(x_i^{\ell+1}-x_i^{\ell})\nonumber\displaybreak[0]\\
=&\sum_{i\in\mathcal{V}}(x_i^{\ell})^T\!\sum_{j\in\mathcal{N}_i}\!\alpha_{\{i,j\}}\phi_{ij}^{\ell}=\sum_{i\in\mathcal{V}}\sum_{j\in\mathcal{N}_i}\dfrac{1}{2}\alpha_{\{i,j\}}(x_i^{\ell}-x_j^{\ell})^T\phi_{ij}^{\ell}.\nonumber
\end{align}
We incorporate this into \eqref{xisi} and then combine the resulting equation with \eqref{eq:deltaV}, leading to
\begin{align}
&\Delta V(\mathbf{x}^\ell)\le\sum_{i\in \mathcal{V}}\left(f_i(x_i^{\ell})-f_i(x_{i}^{\ell+1})+\nabla f_i(x_{i}^{\ell})^T(x_{i}^{\ell+1}-x_{i}^{\ell})\right) \displaybreak[0] \nonumber\\
&+\sum_{i\in \mathcal{V}}\left(\nabla f_i(x_{i}^{\ell+1})-\nabla f_i(x_{i}^{\ell})\right)^T(x_{i}^{\ell+1}-x_{i}^{\ell})\displaybreak[0]\nonumber\\
&+\sum_{i\in\mathcal{V}}\!\sum_{j\in\mathcal{N}_i}\!\!\dfrac{1}{2}\alpha_{\{i,j\}}(x_i^{\ell}\!-\!x_j^{\ell})^T\phi_{ij}^{\ell}\!+\!\sum_{i\in\mathcal{V}}\!\|x_i^{\ell}\!-\!x^*\|\!\cdot\!\|x_i^{\ell+1}\!\!-\!x_i^{\ell}\|\cdot\int_{0}^{1}\|\nabla^2 f_i(x_i^{\ell}\!+\!s(x_i^{\ell+1}\!\!-\!x_i^{\ell}))\!-\!\nabla^2\!f_i(x_i^{\ell})\|ds.\label{eq:deltaVV}
\end{align}
We further bound $(x_i^{\ell}-x_j^{\ell})^T\phi_{ij}^{\ell}$ in \eqref{eq:deltaVV}. To do so, note that $\phi_{ij}^\ell$ defined in \eqref{phiijk1} can be expressed as
\begin{align*}
\phi_{ij}^\ell=\int_0^1\nabla^2 g_{\{i,j\}}(x_i^\ell+s(x_j^\ell-x_i^\ell))(x_j^\ell-x_i^\ell)ds.
\end{align*}
For simplicity, let $A_\phi:=\int_0^1\nabla^2 g_{\{i,j\}}(x_i^\ell+s(x_j^\ell-x_i^\ell))ds$, so that $\phi_{ij}^\ell=A_\phi(x_j^\ell-x_i^\ell)$. Since $g_{\{i,j\}}$ is twice continuously differentiable and locally strongly convex, $A_\phi$ is symmetric positive definite. Also, from the Lipschitz continuity of $\nabla g_{\{i,j\}}$ on $\operatorname{conv}\{\mathcal{C}_i\cup\mathcal{C}_j\}$ which contains $x_i^\ell$ and $x_j^\ell$, we have $\|A_\phi\|\le\Gamma_{\{i,j\}}$ and, thus, $\Gamma_{\{i,j\}}A_\phi\succeq A_\phi^2$. As a result, $\|\phi_{ij}^\ell\|^2=\|A_\phi (x_j^\ell-x_i^\ell)\|^2\le \Gamma_{\{i,j\}}(x_j^\ell-x_i^\ell)^TA_\phi(x_j^\ell-x_i^\ell)=\Gamma_{\{i,j\}}(x_j^\ell-x_i^\ell)^T\phi_{ij}^\ell$, i.e.,
\begin{equation}\label{ieq:Gamma}
(x_i^{\ell}-x_j^{\ell})^T\phi_{ij}^{\ell}\leq-\dfrac{\|\phi_{ij}^{\ell}\|^2}{\Gamma_{\{i,j\}}}.
\end{equation}
By utilizing the strong convexity of $f_i$ on $\mathcal{C}'_i$, the Lipschitz continuity of $\nabla f_i$ and $\nabla^2 f_i$ on $\mathcal{C}'_i$, \eqref{eq:deltaVV}, \eqref{ieq:Gamma}, and \eqref{eq:xx*<=sqrt2Vtheta},
\begin{align}
&\Delta V(\mathbf{x}^\ell)\!\le\!-\!\sum_{i\in\mathcal{V}}\!\frac{\theta'_i}{2}\|x_i^{\ell+1}\!-\!x_i^\ell\|^2\!+\!\sum_{i\in\mathcal{V}}\Theta'_i\|x_i^{\ell+1}\!-\!x_i^\ell\|^2\!-\!\sum_{i\in\mathcal{V}}\!\sum_{j\in\mathcal{N}_i}\!\frac{\alpha_{\{i,j\}}}{2\Gamma_{\{i,j\}}}\|\phi_{ij}^{\ell}\|^2\!+\!\sum_{i\in\mathcal{V}}\frac{L'_i}{2}\sqrt{\frac{2V(\mathbf{x}^0)}{\theta_i}}\|x_i^{\ell+1}\!-\!x_i^\ell\|^2.\label{eq:deltaVnew}
\end{align}
In addition, due to \eqref{eq:xx<=1thetaalphaphi}, we have
\begin{align}
\|x_i^{\ell+1}-x_i^\ell\|^2\le\frac{|\mathcal{N}_i|}{\theta_i^2}\sum_{j\in\mathcal{N}_i}\alpha_{\{i,j\}}^2\|\phi_{ij}^\ell\|^2.\label{eq:xx<=Nithetasumalphaphi}
\end{align}
Since $0<\theta'_i\le\Theta'_i$, we are able to substitute \eqref{eq:xx<=Nithetasumalphaphi} into \eqref{eq:deltaVnew} and derive
\begin{align}
&\Delta V(\mathbf{x}^\ell)\le-\sum_{i\in\mathcal{V}}\sum_{j\in\mathcal{N}_i}\|\phi_{ij}^\ell\|^2\cdot\alpha_{\{i,j\}}\cdot\left[\left(\frac{\theta'_i}{2}-\Theta'_i-\frac{L'_i}{2}\sqrt{\frac{2V(\mathbf{x}^0)}{\theta_i}}\right)\frac{|\mathcal{N}_i|}{\theta_i^2}\alpha_{\{i,j\}}+\frac{1}{2\Gamma_{\{i,j\}}}\right].\label{eq:deltaVfinal}
\end{align}

Due to \eqref{eq:stepsizedrop}, the term inside the bracket in \eqref{eq:deltaVfinal} is guaranteed to be positive. Therefore, $\Delta V(\mathbf{x}^\ell)\le 0$, so that $V(\mathbf{x}^{\ell+1})\le V(\mathbf{x}^\ell)\le V(\mathbf{x}^0)$.

Now we have shown that $V(\mathbf{x}^k)\le V(\mathbf{x}^0)$ $\forall k\ge0$, which suggests $x_i^k\in\mathcal{C}_i$ $\forall i\in\mathcal{V}$ $\forall k\ge0$. Then, similar to the derivation of \eqref{eq:deltaVnew}, we can see from \eqref{eq:deltaVV}, \eqref{ieq:Gamma}, \eqref{eq:xx*<=sqrt2Vtheta}, the strong convexity of $f_i$ on $\operatorname{conv}\mathcal{C}_i$, and the Lipschitz continuity of $\nabla f_i$ and $\nabla^2 f_i$ on $\operatorname{conv}\mathcal{C}_i$ that
\begin{align}
\Delta V(\mathbf{x}^\ell)\le&-\sum_{i\in\mathcal{V}}\frac{\theta_i}{2}\|x_i^{\ell+1}-x_i^\ell\|^2+\sum_{i\in\mathcal{V}}\Theta_i\|x_i^{\ell+1}-x_i^\ell\|^2\nonumber\displaybreak[0]\\
&-\sum_{i\in\mathcal{V}}\sum_{j\in\mathcal{N}_i}\frac{\alpha_{\{i,j\}}}{2\Gamma_{\{i,j\}}}\|\phi_{ij}^{\ell}\|^2+\sum_{i\in\mathcal{V}}\frac{L_i}{2}\sqrt{\frac{2V(\mathbf{x}^0)}{\theta_i}}\|x_i^{\ell+1}-x_i^\ell\|^2.\label{eq:deltaVnewtighter}
\end{align}
This, along with \eqref{eq:xx<=Nithetasumalphaphi}, implies that \eqref{eq:deltaVfinal} with $\Theta'_i$, $\theta'_i$, and $L'_i$ replaced by $\Theta_i$, $\theta_i$, and $L_i$, respectively, still holds, whose right-hand side is also nonpositive because of \eqref{eq:stepsizedrop} and because $\Theta_i\le\Theta'_i$, $\theta_i\ge\theta'_i$, and $L_i\le L'_i$. Moreover, since $g_{\{i,j\}}$ $\forall\{i,j\}\in\mathcal{E}$ are strongly convex on $\operatorname{conv}\{\mathcal{C}_i\cup\mathcal{C}_j\}$, we have $\|\phi_{ij}^k\|\cdot\|x_j^k-x_i^k\|\ge (\phi_{ij}^k)^T(x_j^k-x_i^k)\ge\gamma_{\{i,j\}}\|x_j^k-x_i^k\|^2$, which gives 
\begin{align*}
\|\phi_{ij}^k\|^2\ge\gamma_{\{i,j\}}^2\|x_j^k-x_i^k\|^2.
\end{align*}
It follows that \eqref{eq:Vdrop} holds.

\subsection{Proof of Theorem~\ref{theo:firsttheo}}\label{sec:proofoftheo:firsttheo}

Note from Lemma~\ref{lem:monoto} that $V(\mathbf{x}^k)$ is non-increasing. In addition, since $V(\mathbf{x}^k)$ is bounded from below by $0$, $\lim_{k\rightarrow\infty}V(\mathbf{x}^k)$ exists. This, along with \eqref{eq:Vdrop}, implies that $\sum_{k=0}^\infty\sum_{\{i,j\}\in\mathcal{E}} \|x_i^{k}-x_j^k\|^2$ is finite. It then follows that for any $\{i,j\}\in\mathcal{E}$, $\lim_{k\rightarrow\infty}\|x_i^k-x_j^k\|=0$. Since $\mathcal{G}$ is connected, this leads to \eqref{eq:consensuserrorbound}.

\subsection{Proof of Theorem~\ref{thm:convepsilon}}\label{sec:proofofthm:convergence}

First of all, note that the right-hand side of \eqref{eq:stepsizedrop} is equal to $\frac{1}{\Gamma_{\{i,j\}}}\min\{\frac{1}{\tilde{\eta}_i},\frac{1}{\tilde{\eta}_j}\}$. Hence, \eqref{eq:stepsizeepsilon} suffices to meet \eqref{eq:stepsizedrop}, so that Lemma~\ref{lem:monoto} holds here and $x_i^k\in\mathcal{C}_i$ $\forall i\in\mathcal{V}$ $\forall k\ge0$. Below, we provide an upper bound on $\|\mathbf{x}^k-\mathbf{x}^*\|$ which goes below $\epsilon$ as $k\rightarrow\infty$. To this end, define $\bar{x}^k=\frac{1}{N}\sum_{i\in\mathcal{V}}x_i^k$ as the average of $x_i^k$ $\forall i\in\mathcal{V}$ and define $\bar{\mathbf{x}}^k=[\bar{x}^k;\dots;\bar{x}^k]\in\mathbb{R}^{nN}$ for each $k\ge0$. Note that
\begin{align}
\|\mathbf{x}^k-\mathbf{x}^*\|&\le\|\bar{\mathbf{x}}^k-\mathbf{x}^*\|+\|\bar{\mathbf{x}}^k-\mathbf{x}^k\|\nonumber\displaybreak[0]\\
&=\sqrt{N}\|\bar{x}^k-x^*\|+\|\bar{\mathbf{x}}^k-\mathbf{x}^k\|. \label{eq:xx<=sqrtNxx+xx}
\end{align}
On the other hand, note that $\bar{x}^k\in\operatorname{conv}\cup_{i\in\mathcal{V}}\mathcal{C}_i$ and $x^*\in\cap_{i\in\mathcal{V}}\mathcal{C}_i$. Hence, using the local strong convexity of each $f_i$,
\begin{align*}
&\|\sum_{i\in\mathcal{V}}\big(\nabla f_i(\bar{x}^k)-\nabla f_i(x^*)\big)\|\cdot\|\bar{x}^k-x^*\|\displaybreak[0]\\
\ge&\sum_{i\in\mathcal{V}}\big(\nabla f_i(\bar{x}^k)-\nabla f_i(x^*)\big)^T(\bar{x}^k-x^*)\ge\sum_{i\in\mathcal{V}}\bar{\theta}_i\|\bar{x}^k-x^*\|^2.
\end{align*}
Since $\sum_{i\in\mathcal{V}}\nabla f_i(x^*)=0$, the above inequality gives
\begin{align*}
\|\bar{x}^k-x^*\|\le\frac{1}{\sum_{i\in\mathcal{V}}\bar{\theta}_i}\|\sum_{i\in\mathcal{V}}\nabla f_i(\bar{x}^k)\|.
\end{align*}
It then follows from \eqref{eq:xx<=sqrtNxx+xx} that
\begin{align}
\|\mathbf{x}^k-\mathbf{x}^*\|\le\frac{\sqrt{N}}{\sum_{i\in\mathcal{V}}\!\bar{\theta}_i}\|\sum_{i\in\mathcal{V}}\!\nabla f_i(\bar{x}^k)\|+\|\bar{\mathbf{x}}^k-\mathbf{x}^k\|.\label{eq:optimaldist}
\end{align}
Note that 
\begin{align}
\|\!\sum_{i\in\mathcal{V}}\!\nabla\! f_i(\bar{x}^k)\|&\le\!\|\!\sum_{i\in\mathcal{V}}\!\!\big(\nabla\!f_i(\bar{x}^k)\!-\!\nabla\! f_i(x_i^k)\big)\|\!+\!\|\!\sum_{i\in\mathcal{V}}\!\nabla\!f_i(x_i^k)\|\displaybreak[0]\nonumber\\
&\leq\sum_{i\in\mathcal{V}}\bar{\Theta}_i\|\bar{x}^k-{x}_i^k\|+\!\|\sum_{i\in\mathcal{V}}\!\nabla f_i(x_i^k)\|\displaybreak[0]\nonumber\\
&\leq\max_{i\in\mathcal{V}}\bar{\Theta}_i\sqrt{N}\|\bar{\mathbf{x}}^k-\mathbf{x}^k\|+\!\|\sum_{i\in\mathcal{V}}\!\nabla f_i(x_i^k)\|.\nonumber
\end{align}
Combining this with \eqref{eq:optimaldist} results in
\begin{align}
&\|\mathbf{x}^k-\mathbf{x}^*\|\le\|\bar{\mathbf{x}}^k-\mathbf{x}^k\|+\frac{\sqrt{N}}{\sum_{i\in\mathcal{V}}\!\bar{\theta}_i}(\max_{i\in\mathcal{V}}\bar{\Theta}_i\sqrt{N}\|\bar{\mathbf{x}}^k-\mathbf{x}^k\|+\!\|\sum_{i\in\mathcal{V}}\!\nabla f_i(x_i^k)\|).\label{eq:xx<=xx+NthetaNThetaxx+sumnablaf}
\end{align}
Subsequently, we bound $\|\sum_{i\in\mathcal{V}}\!\nabla f_i(x_i^k)\|$ in \eqref{eq:xx<=xx+NthetaNThetaxx+sumnablaf}. From Lemma~\ref{lem:monoto}, $x_i^t$ $\forall t\ge0$ are contained in $\mathcal{C}_i$. Then, due to \eqref{eq:xi0}, \eqref{eq:gradientsumchange}, and the local Lipschitz continuity of $\nabla^2 f_i$,
\begin{align}
\|\sum_{i\in \mathcal{V}}\nabla f_i(x_i^k)\|&=\|\sum_{i\in \mathcal{V}}\big(\nabla f_i(x_i^k)-\nabla f_i(x_i^0)\big)\|\nonumber\displaybreak[0]\\
&\le\sum_{t=0}^{k-1}\|\sum_{i\in \mathcal{V}}\big(\nabla\! f_i(x_i^{t+1})-\nabla f_i(x_i^t)\big)\|\nonumber\displaybreak[0]\\
&\le\sum_{t=0}^{k-1}\sum_{i\in\mathcal{V}}\dfrac{L_i}{2}\|x_i^{t+1}-x_i^t\|^2.\label{eq:sumnablaf<=sumsumL2xx}
\end{align}
Besides, due to \eqref{eq:xx<=Nithetasumalphaphi},
\begin{align}
-\sum_{i\in\mathcal{V}}\sum_{j\in\mathcal{N}_i}\dfrac{\alpha_{\{i,j\}}}{2\Gamma_{\{i,j\}}}\|\phi_{ij}^k\|^2&\leq\!-\!\!\sum_{i\in\mathcal{V}}\!\dfrac{\theta_i^2}{2|\mathcal{N}_i|\!\max_{j\in\mathcal{N}_i}\!\big\{\alpha_{\{i,j\}}\Gamma_{\{i,j\}}\big\}}\!\cdot\!\dfrac{|\mathcal{N}_i|}{\theta_i^2}\!\!\sum_{j\in\mathcal{N}_i}\!\!\alpha_{\{i,j\}}^2\|\phi_{ij}^k\|^2\!\nonumber\\
&\le-\!\sum_{i\in\mathcal{V}}\!\dfrac{\theta_i^2}{2|\mathcal{N}_i|\max_{j\in\mathcal{N}_i}\!\big\{\alpha_{\{i,j\}}\Gamma_{\{i,j\}}\big\}}\|x_i^{k+1}\!-\!x_i^k\|^2.\nonumber
\end{align}
This, together with \eqref{eq:deltaVnewtighter}, implies that
\begin{align*}
V(\mathbf{x}^{k+1})-V(\mathbf{x}^k)\le-\sum_{i\in\mathcal{V}}\rho_i\|x_i^{k+1}-x_i^k\|^2,
\end{align*}
where, due to \eqref{eq:stepsizedrop}, $\rho_i$ given below is positive: 
\begin{equation}\label{eq:rhoi}
\rho_i=\dfrac{\theta_i^2}{2|\mathcal{N}_i|\!\max_{j\in\mathcal{N}_i}\!\big\{\alpha_{\{i,j\}}\Gamma_{\{i,j\}}\big\}}+\frac{\theta_i}{2}-\Theta_i-\frac{L_i}{2}\sqrt{\dfrac{2V(\mathbf{x}^0)}{\theta_i}}.
\end{equation}
It then follows from \eqref{eq:sumnablaf<=sumsumL2xx} that
\begin{align}
&\|\sum_{i\in \mathcal{V}}\nabla f_i(x_i^k)\|\le\sum_{t=0}^{k-1}\max_{i\in\mathcal{V}}\dfrac{L_i}{2\rho_i}\sum_{i\in\mathcal{V}}\rho_i\|x_i^{t+1}-x_i^t\|^2\le\max_{i\in\mathcal{V}}\dfrac{L_i}{2\rho_i}(V(\mathbf{x}^0)-V(\mathbf{x}^k))\le\max_{i\in\mathcal{V}}\dfrac{L_i}{2\rho_i}V(\mathbf{x}^0).\label{eq:gradientsumbound}
\end{align}
Now by substituting \eqref{eq:gradientsumbound} into \eqref{eq:xx<=xx+NthetaNThetaxx+sumnablaf},
\begin{align}
\|\mathbf{x}^k\!-\!\mathbf{x}^*\|\le&\max_{i\in\mathcal{V}}\dfrac{L_i}{\rho_i}\cdot\dfrac{\sqrt{N}}{2\sum_{i\in\mathcal{V}}\bar{\theta}_i}V(\mathbf{x}^0)+\Bigl(\dfrac{N\max_{i\in\mathcal{V}}\bar{\Theta}_i}{\sum_{i\in\mathcal{V}}\bar{\theta}_i}+1\Bigr)\|\bar{\mathbf{x}}^k-\mathbf{x}^k\|\label{ieq:consensusoptimal}.
\end{align}
From Theorem~\ref{theo:firsttheo}, we have $\|\bar{\mathbf{x}}^k-\mathbf{x}^k\|\rightarrow0$ as $k\rightarrow\infty$. Also, \eqref{eq:stepsizeepsilon} guarantees $\max_{i\in\mathcal{V}} \frac{L_i}{\rho_i}\cdot\frac{\sqrt{N}}{2\sum_{i\in\mathcal{V}}\!\bar{\theta}_i}V(\mathbf{x}^0)<\epsilon$. Consequently, $\lim_{k\rightarrow\infty}\|\mathbf{x}^k-\mathbf{x}^*\|<\epsilon$.

\subsection{Proof of Theorem~\ref{thm:rate}}\label{sec:proofofthm:consensuserror}

When $g_{\{i,j\}}(x)=\frac{1}{2}x^Tx$ $\forall \{i,j\}\in\mathcal{E}$, \eqref{scheme} becomes
\begin{align*}
x_i^{k+1}=x_i^k-(\nabla^2 f_i(x_i^k))^{-1}\sum_{j\in \mathcal{N}_i}\alpha_{\{i,j\}}(x_i^k-x_j^k),\quad\forall i\in\mathcal{V},
\end{align*}
or equivalently,
\begin{equation}\label{eq:secmatrix}
\mathbf{x}^{k+1}=\mathbf{x}^k-(\nabla^2F(\mathbf{x}^k))^{-1}(H_{\mathcal{G}}\otimes I_n)\mathbf{x}^k,\quad\forall k\ge0,
\end{equation}
where $H_{\mathcal{G}}$ is equal to $\hat{H}_\mathcal{G}$ in \eqref{matrix:H_G} with $h_{\{i,j\}}=\alpha_{\{i,j\}}$ $\forall\{i,j\}\in\mathcal{E}$. Hence, $H_{\mathcal{G}}=H_{\mathcal{G}}^T$ is positive semidefinite and its rank is $N-1$. Let $S=\{\mathbf{x}=[x_1;\ldots;x_N]\in\mathbb{R}^{nN}:x_1=\dots=x_N\}$ and $S^{\perp}=\{\mathbf{x}=[x_1;\ldots;x_N]\in\mathbb{R}^{nN}:x_1+\dots+x_N=0\}$ be the orthogonal complement of $S$. Then, it can be shown that
\begin{align}
&\text{range}(H_{\mathcal{G}}\otimes I_n)=\text{range}(H_{\mathcal{G}}^{\frac{1}{2}}\otimes I_n)=\text{range}(H_{\mathcal{G}}^{\dagger}\otimes I_n)=S^{\perp}.\label{eq:rangeH} 
\end{align}

Below, we first show that $(\mathbf{x}^k)_{k=0}^\infty$ is a Cauchy sequence. To do so, let $\mathbf{y}^k=(H_{\mathcal{G}}^{\frac{1}{2}}\otimes I_n)\mathbf{x}^k$. Multiplying $(H_{\mathcal{G}}^{\frac{1}{2}}\otimes I_n)$ on both sides of \eqref{eq:secmatrix} yields
\begin{equation*}
\mathbf{y}^{k+1}=\mathbf{y}^k-(H_{\mathcal{G}}^{\frac{1}{2}}\otimes I_n)(\nabla^2F(\mathbf{x}^k))^{-1}(H_{\mathcal{G}}^{\frac{1}{2}}\otimes I_n)\mathbf{y}^k.
\end{equation*}
Moreover, since $\mathbf{y}^k\in S^{\perp}$, we have $\mathbf{y}^k=(H_{\mathcal{G}}H_{\mathcal{G}}^{\dagger}\otimes I_n)\mathbf{y}^k$. It follows that
\begin{align}
&\mathbf{y}^{k+1}=\big[(H_{\mathcal{G}}H_{\mathcal{G}}^{\dagger}\otimes I_n)-(H_{\mathcal{G}}^{\frac{1}{2}}\otimes I_n)(\nabla^2F(\mathbf{x}^k))^{-1}(H_{\mathcal{G}}^{\frac{1}{2}}\otimes I_n)\big]\mathbf{y}^k.\label{eq:ymatrix}
\end{align}
We now bound the term in the bracket of \eqref{eq:ymatrix}. For any $\mathbf{x}\in\mathbb{R}^{nN}$, let $\mathbf{z}=P_{S^\perp}(\mathbf{x})$. Due to \eqref{eq:rangeH},
\begin{align}
&\mathbf{x}^T\!\big[(H_{\mathcal{G}}H_{\mathcal{G}}^{\dagger}\!\otimes\!I_n)\!-\!(H_{\mathcal{G}}^{\frac{1}{2}}\!\otimes\!I_n)(\nabla^2\!F(\mathbf{x}^k))^{-1}\!(H_{\mathcal{G}}^{\frac{1}{2}}\!\otimes\!I_n)\big]\mathbf{x}\displaybreak[0]\nonumber\\
=&\mathbf{z}^T\big[I_N\otimes I_n\!-\!(H_{\mathcal{G}}^{\frac{1}{2}}\!\otimes\!I_n)(\nabla^2F(\mathbf{x}^k))^{-1}(H_{\mathcal{G}}^{\frac{1}{2}}\!\otimes I_n)\big]\mathbf{z}.\displaybreak[0]\label{ieq:norminitial}
\end{align}
Further, since $(H_{\mathcal{G}}^{\frac{1}{2}}\otimes I_n)(\nabla^2F(\mathbf{x}^k))^{-1}(H_{\mathcal{G}}^{\frac{1}{2}}\otimes I_n)$ is positive semidefinite and $z\in S^{\perp}$, we have
\begin{align} 
&\big(1-\lambda_{\max}((H_{\mathcal{G}}^{\frac{1}{2}}\otimes I_n)(\nabla^2F(\mathbf{x}^k))^{-1}(H_{\mathcal{G}}^{\frac{1}{2}}\otimes I_n))\big)\|\mathbf{z}\|^2\displaybreak[0]\nonumber\\ 
\leq&\mathbf{z}^T\big[I_N\otimes I_n-(H_{\mathcal{G}}^{\frac{1}{2}}\otimes I_n)(\nabla^2F(\mathbf{x}^k))^{-1}(H_{\mathcal{G}}^{\frac{1}{2}}\otimes I_n)\big]\mathbf{z}\displaybreak[0]\nonumber\\
\leq&\big(1-\lambda_{2}((H_{\mathcal{G}}^{\frac{1}{2}}\otimes I_n)(\nabla^2F(\mathbf{x}^k))^{-1}(H_{\mathcal{G}}^{\frac{1}{2}}\otimes I_n))\big)\|\mathbf{z}\|^2.\displaybreak[0]\label{ieq:normmedia}
\end{align}
Also note that $\|\mathbf{z}\|\leq\|\mathbf{x}\|$. It then follows from \eqref{ieq:norminitial} and \eqref{ieq:normmedia} that for any $\mathbf{x}\in\mathbb{R}^{nN}$,
\begin{align*}
&\big|\mathbf{x}^T\big[(H_{\mathcal{G}}H_{\mathcal{G}}^{\dagger}\otimes I_n)\!-\!(H_{\mathcal{G}}^{\frac{1}{2}}\!\otimes I_n)(\nabla^2F(\mathbf{x}^k))^{-1}(H_{\mathcal{G}}^{\frac{1}{2}}\otimes I_n)\big]\mathbf{x}\big|\displaybreak[0]\\
&\leq\max\big\{\lambda_{\max}((H_{\mathcal{G}}^{\frac{1}{2}}\otimes I_n)(\nabla^2F(\mathbf{x}^k))^{-1}(H_{\mathcal{G}}^{\frac{1}{2}}\otimes I_n))-1, 1-\lambda_{2}((H_{\mathcal{G}}^{\frac{1}{2}}\otimes I_n)(\nabla^2F(\mathbf{x}^k))^{-1}(H_{\mathcal{G}}^{\frac{1}{2}}\otimes I_n))\big\}\|\mathbf{x}\|^2.
\end{align*} 
From Lemma~\ref{lem:monoto}, $\mathbf{x}^k\in\mathcal{C}_1\times\cdots\times\mathcal{C}_N$. Therefore,
\begin{align}
&\|(H_{\mathcal{G}}H_{\mathcal{G}}^{\dagger}\otimes I_n)\!-\!(H_{\mathcal{G}}^{\frac{1}{2}}\otimes\!I_n)(\nabla^2F(\mathbf{x}^k))^{-1}(H_{\mathcal{G}}^{\frac{1}{2}}\!\otimes I_n)\|\displaybreak[0]\nonumber\\
\leq&\max\big\{\lambda_{\max}((H_{\mathcal{G}}^{\frac{1}{2}}\otimes I_n)(\nabla^2F(\mathbf{x}^k))^{-1}(H_{\mathcal{G}}^{\frac{1}{2}}\otimes I_n))-1,1-\lambda_{2}((H_{\mathcal{G}}^{\frac{1}{2}}\otimes I_n)(\nabla^2F(\mathbf{x}^k))^{-1}(H_{\mathcal{G}}^{\frac{1}{2}}\otimes I_n))\big\}\displaybreak[0]\nonumber\\
\leq&\max\Big\{\dfrac{\lambda_{\max}(H_\mathcal{G})}{\theta}-1,1-\dfrac{\lambda_{2}(H_{\mathcal{G}})}{\Theta}\Big\}.\label{eq:Hnormbound}
\end{align}
Furthermore, note that $\lambda_{2}(H_{\mathcal{G}})$ and $\lambda_{\max}(H_{\mathcal{G}})$ in the above inequality can be bounded as follows:
\begin{align}
&\lambda_{2}(H_\mathcal{G})\geq\min_{\{i,j\}\in\mathcal{E}}\alpha_{\{i,j\}}\lambda_{2}(L_{\mathcal{G}}),\nonumber\displaybreak[0]\\
&\lambda_{\max}(H_\mathcal{G})\leq\max_{\{i,j\}\in\mathcal{E}}\alpha_{\{i,j\}}\lambda_{\max}(L_{\mathcal{G}}),\label{eq:upperbound}
\end{align}
where $L_{\mathcal{G}}$ is the Laplacian matrix of $\mathcal{G}$. Incorporating these bounds into \eqref{eq:Hnormbound} yields
\begin{equation}
\|(H_{\mathcal{G}}H_{\mathcal{G}}^{\dagger}\!\otimes\! I_n)\!-\!(H_{\mathcal{G}}^{\frac{1}{2}}\!\otimes\! I_n)(\nabla^2\!F(\mathbf{x}^k))^{-1}\!(H_{\mathcal{G}}^{\frac{1}{2}}\!\otimes\!I_n)\|\!\le\!q,\label{ieq:q}
\end{equation}
where $q$ is given in the theorem statement. Besides, because of the condition $\alpha_{\{i,j\}}<\theta/(\max_{i\in\mathcal{V}}|\mathcal{N}_i|)$ $\forall\{i,j\}\in\mathcal{E}$ and because $\lambda_{\max}(L_{\mathcal{G}})\leq\min\{N,2\max_{i\in\mathcal{V}}|\mathcal{N}_i|\}$, we have $\max_{\{i,j\}\in\mathcal{E}}\alpha_{\{i,j\}}\lambda_{\max}(L_\mathcal{G})/\theta<2$ and, thus, $0<q<1$.
Combining \eqref{ieq:q} with \eqref{eq:ymatrix} implies that $\|\mathbf{y}^{k+1}\|\leq q\|\mathbf{y}^k\|$. This, along with \eqref{eq:secmatrix}, indicates that
\begin{align}
&\|\mathbf{x}^{k+1}-\mathbf{x}^k\|=\|(\nabla^2F(\mathbf{x}^k))^{-1}(H_{\mathcal{G}}^{\frac{1}{2}}\otimes I_n)\mathbf{y}^k\|\displaybreak[0]\nonumber\\
&\le\dfrac{\lambda_{\max}(H_{\mathcal{G}}^{\frac{1}{2}})}{\theta}\|\mathbf{y}^0\|{q^k}\leq\dfrac{\lambda_{\max}(H_{\mathcal{G}})}{\theta}\|\mathbf{x}^0\|{q^k}.\displaybreak[0]\label{ieq:rate}
\end{align}
As a result, given any $\epsilon>0,$ there exists $\bar{K}_\epsilon=\lceil\log_{q} \frac{\epsilon\theta(1-q)}{\lambda_{\max}(H_{\mathcal{G}})\|\mathbf{x}^0\|}\rceil$ such that $\forall k_1\geq k_2\geq \bar{K}_\epsilon$,
\begin{align*}
&\|\mathbf{x}^{k_1}-\mathbf{x}^{k_2}\|\leq\sum_{l=k_2}^{k_1-1}\|\mathbf{x}^{l+1}-\mathbf{x}^l\|\displaybreak[0]\\
&\leq\dfrac{\lambda_{\max}(H_{\mathcal{G}})}{\theta}\|\mathbf{x}^0\|\sum_{l=k_2}^{k_1-1}{q^l}\leq\dfrac{\lambda_{\max}(H_{\mathcal{G}})}{\theta(1-q)}\|\mathbf{x}^0\|q^{\bar{K}_\epsilon}\leq\epsilon.\displaybreak[0]
\end{align*}
Therefore, $(\mathbf{x}^k)_{k=0}^\infty$ is a Cauchy sequence and is convergent.

From Theorem~\ref{theo:firsttheo}, we know that $\lim_{k\rightarrow\infty}x_i^k=\tilde{x}$ $\forall i\in\mathcal{V}$ for some $\tilde{x}\in\mathbb{R}^n$. Hence, $\lim_{k\rightarrow\infty}\mathbf{x}^k=\tilde{\mathbf{x}}:=[\tilde{x};\dots;\tilde{x}]\in\mathbb{R}^{nN}$. Then, due to \eqref{ieq:rate}, 
\begin{align}
\|\tilde{\mathbf{x}}-\mathbf{x}^k\|&=\|\lim_{t\to\infty}(\mathbf{x}^t-\mathbf{x}^k)\|=\|\lim_{t\to\infty}\!\sum_{l=k}^{t-1}(\mathbf{x}^{l+1}-\mathbf{x}^l)\|\nonumber\displaybreak[0]\\
&\leq\sum_{l=k}^{\infty}\|\mathbf{x}^{l+1}-\mathbf{x}^l\|\leq\dfrac{\lambda_{\max}(H_{\mathcal{G}})}{\theta} \|\mathbf{x}^0\| \sum_{l=k}^{\infty}{q^l}\nonumber\displaybreak[0]\\
&\leq\dfrac{\lambda_{\max}(H_{\mathcal{G}})}{\theta(1-q)}\|\mathbf{x}^0\|q^k,\quad\forall k\ge0.\nonumber
\end{align}
This and \eqref{eq:upperbound} lead to \eqref{eq:xx<=Cqk}.

Finally, we provide a bound on $\|\tilde{\mathbf{x}}-\mathbf{x}^*\|$. Again, from Theorem~\ref{theo:firsttheo}, $\lim_{k\to\infty}\|\mathbf{x}^k-\bar{\mathbf{x}}^k\|=0$, where $\bar{\mathbf{x}}^k=[\bar{x}^k;\ldots;\bar{x}^k]\in\mathbb{R}^{nN}$ and $\bar{x}^k=\frac{1}{N}\sum_{i\in\mathcal{V}}x_i^k$. It follows from \eqref{ieq:consensusoptimal} that
\begin{align*}
\|\tilde{\mathbf{x}}-\mathbf{x}^*\|\leq\max_{i\in\mathcal{V}}\dfrac{L_i}{\rho_i}\cdot\dfrac{\sqrt{N}}{2\sum_{i\in\mathcal{V}}\bar{\theta}_i}V(\mathbf{x}^0),
\end{align*}
where $\rho_i>0$ is given by \eqref{eq:rhoi}. Moreover, for each $\{i,j\}\in\mathcal{E}$, $\Gamma_{\{i,j\}}=1$ as $g_{\{i,j\}}=\frac{1}{2}x^Tx$, so that \eqref{eq:optimalityerror} holds.


\bibliographystyle{IEEEtran}
\bibliography{reference}

\begin{IEEEbiography}
	[{\includegraphics[width=1in,height=1.25in,clip,keepaspectratio]{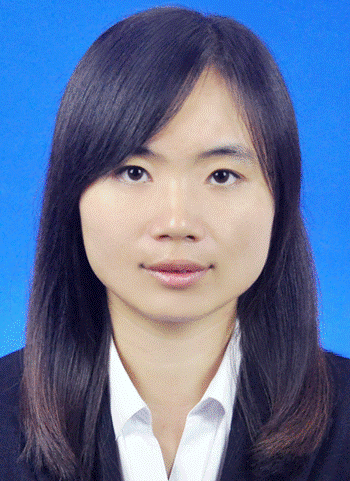}}]
{Hejie Wei} received the Ph.D. degree in Computational Mathematics from Fudan University, Shanghai, China, in 2017. From 2017 to 2019, She was a postdoctoral researcher in the School of Information Science and Technology at ShanghaiTech University, Shanghai, China. Since 2019, She has been a lecturer in the school of Statistics and Mathematics at Shanghai Lixin University of Accounting and Finance. Her research interests include distributed optimization and numerical optimization. 
\end{IEEEbiography}
\begin{IEEEbiography}
	[{\includegraphics[width=1in,height=1.25in,clip,keepaspectratio]{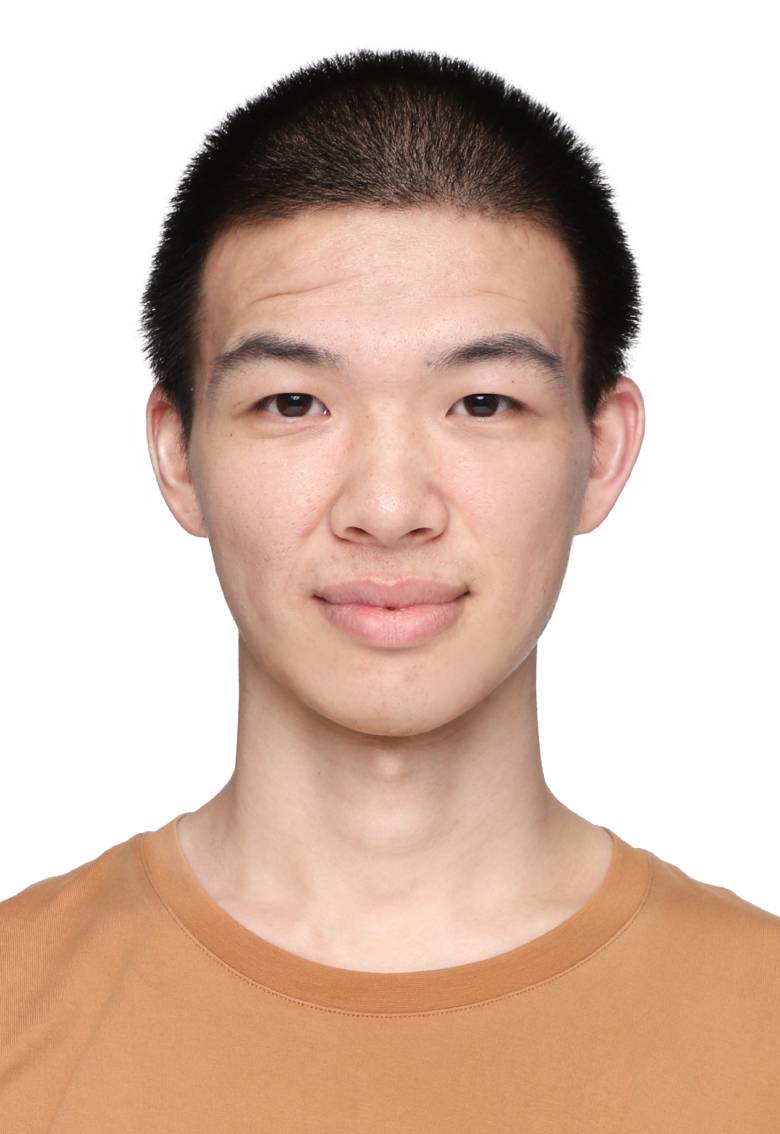}}]{Zhihai Qu} received the B.S. degree in Electronic Information Science and Technology from Hefei University of Technology, Hefei, China, in 2017. He is currently pursuing his Master degree in the School of Information Science and Technology at ShanghaiTech University, Shanghai, China. His research interests include distributed optimization and large-scale optimization algorithms.
\end{IEEEbiography}

\begin{IEEEbiography}
	[{\includegraphics[width=1in,height=1.25in,clip,keepaspectratio]{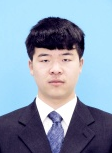}}]{Xuyang Wu}
	(SM'17) received the B.S. degree in Information and Computing Science from Northwestern Polytechnical University, Xi'an, China, in 2015. He is currently pursuing his Ph.D. degree in the School of Information Science and Technology at ShanghaiTech University, Shanghai, China. His research interests include distributed optimization and large-scale optimization algorithms.
	
\end{IEEEbiography}

\begin{IEEEbiography}
	[{\includegraphics[width=1in,height=1.25in,clip,keepaspectratio]{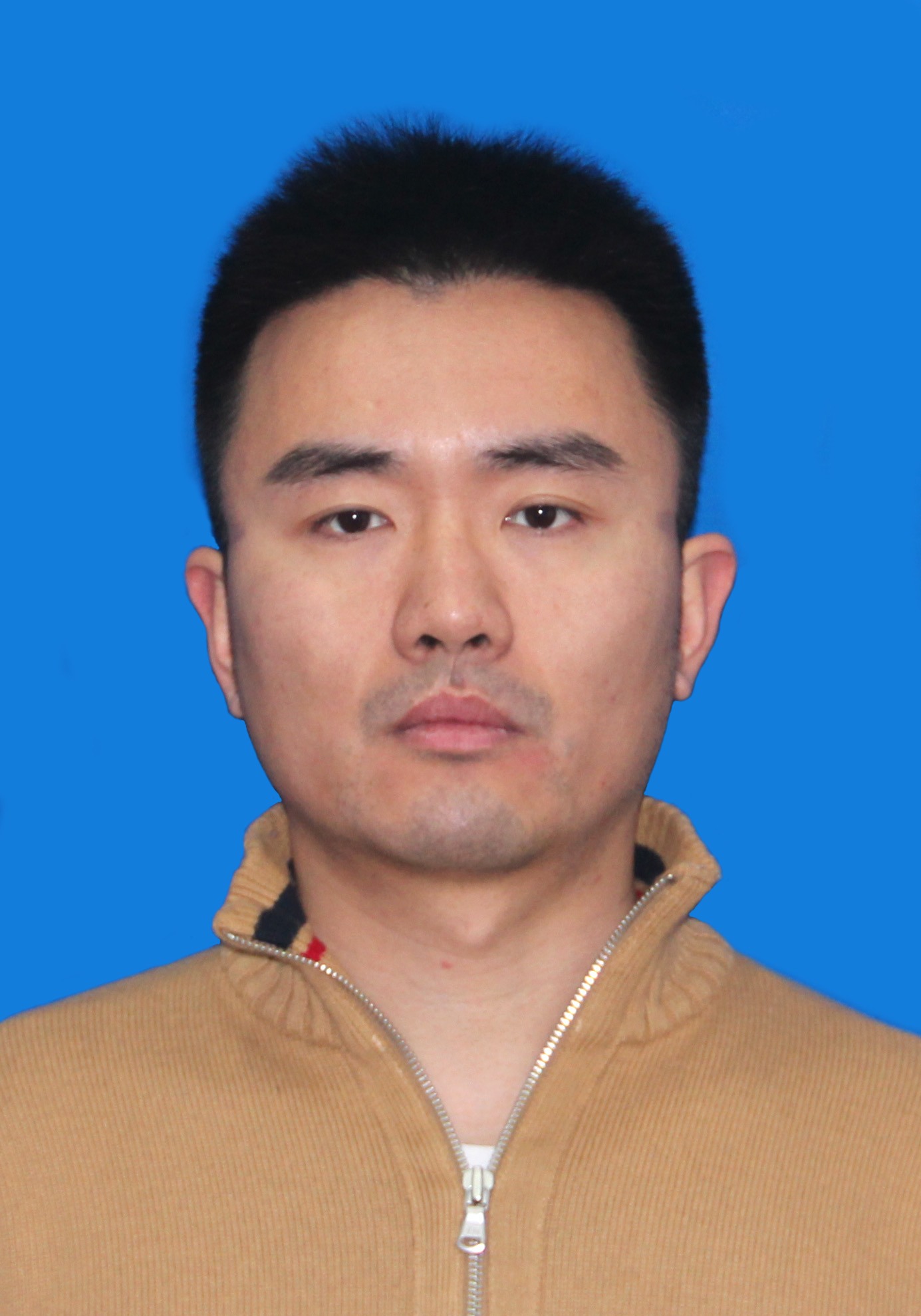}}] {Hao Wang} (M'18) is currently an Assistant Professor with the School of Information Science and Technology at ShanghaiTech University, Shanghai, China. He received the B.S. and M.S. degrees in mathematics from Beihang University, Beijing, China, and the Ph.D. degree from the Industrial and Systems Engineering Department, Lehigh University, Bethlehem, PA, USA in 2015. He was with ExxonMobil Corporate Strategic Research Lab, Mitsubishi Electric Research Lab, and GroupM R$\&$D. His interests include nonlinear optimization and machine learning.
	
\end{IEEEbiography}

\begin{IEEEbiography}[{\includegraphics[width=1in,height=1.25in,clip,keepaspectratio]{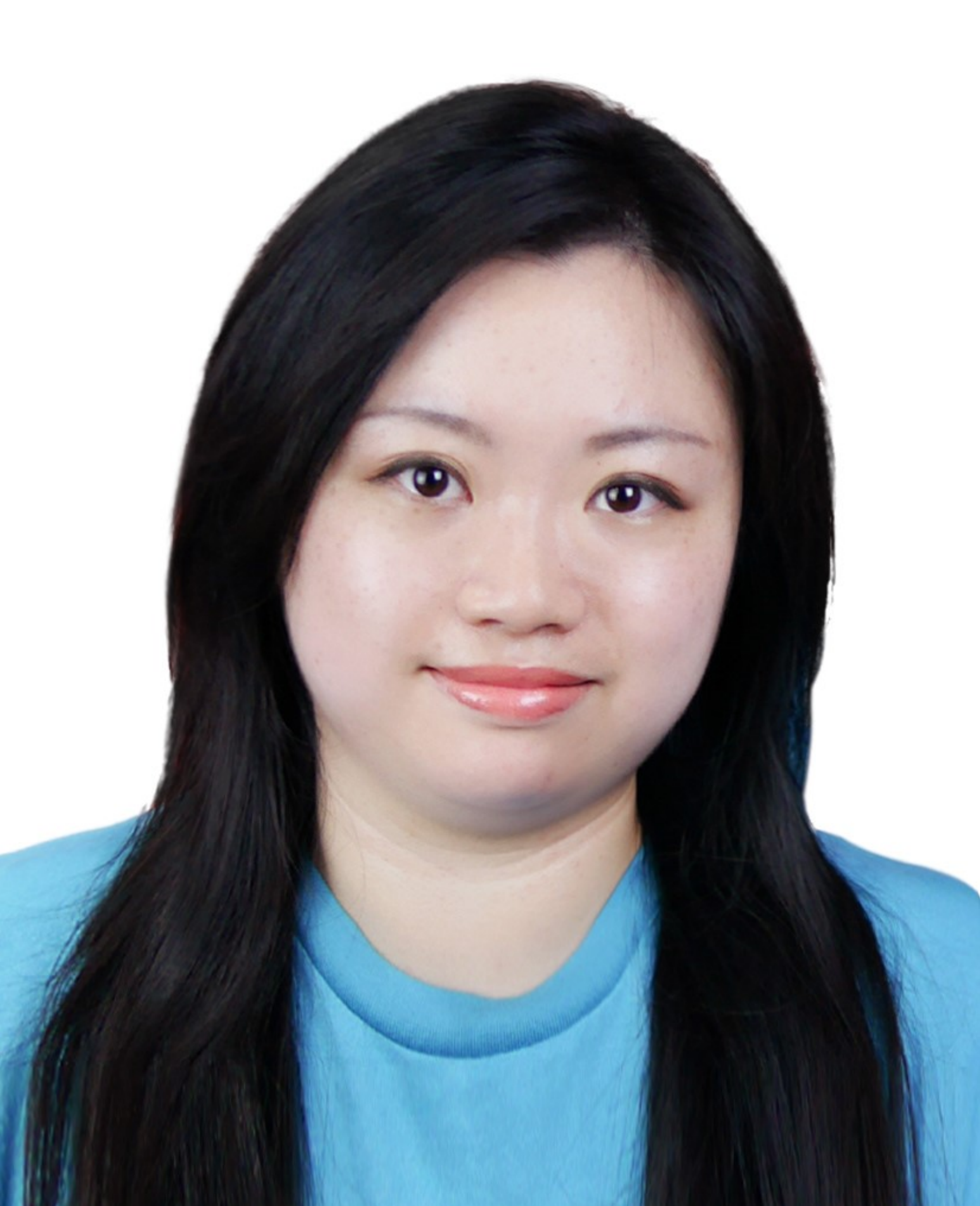}}]{Jie Lu} (SM'08-M'13)
	received the B.S. degree in Information Engineering from Shanghai Jiao Tong University, China, in 2007, and the Ph.D. degree in Electrical and Computer Engineering from the University of Oklahoma, USA, in 2011. From 2012 to 2015 she was a postdoctoral researcher with KTH Royal Institute of Technology, Stockholm, Sweden, and with Chalmers University of Technology, Gothenburg, Sweden. Since 2015, she has been an assistant professor in the School of Information Science and Technology at ShanghaiTech University, Shanghai, China. Her research interests include distributed optimization, optimization theory and algorithms, and networked dynamical systems.
\end{IEEEbiography}

\end{document}